\documentclass[12pt, leqno]{amsart}
\usepackage[letterpaper, margin = 1.3in]{geometry}
\usepackage{amsthm,amsmath,amsfonts,pifont,dsfont,mathtools,bbm}

\RequirePackage[colorlinks,citecolor=blue,urlcolor=blue]{hyperref}
\usepackage{tikz}
\usetikzlibrary{arrows}
\usepackage{enumitem}
\usepackage{relsize}
\usepackage{graphicx}
\usepackage{dsfont}

\makeatletter
\newcommand{\leqn}{\tagsleft@true}
\newcommand{\reqn}{\tagsleft@false}
\makeatother

 \mathtoolsset{showonlyrefs}

\setenumerate[1]{label=\thesection.\arabic*.}
\setenumerate[2]{label*=\arabic*.}

\newcommand{\ex}{\mathbb{E}}
\newcommand{\prob}{\mathbb{P}}

\newcommand{\e}{\epsilon}
\newcommand{\ds}{\displaystyle}
\newcommand{\N}{\mathbb{N}}

\newcommand{\R}{\mathbb{R}}

\newcommand{\wt}[1]{\widetilde{#1}}
\newcommand{\lra}{\longrightarrow}
\newcommand{\mb}[1]{\mathbb{#1}}
\newcommand{\mr}[1]{\mathrm{#1}}
\newcommand{\mc}[1]{\mathcal{#1}}
\newcommand{\ld}{\lambda}
\renewcommand{\d}{\delta}
\newcommand{\F}{\mathcal{F}}
\newcommand{\md}{\mathrm{d}}

\newcommand{\dist}{\overset{d}{=}}
\newcommand{\ol}[1]{\overline{#1}}

\renewcommand{\L}{\mb{L}}

\usepackage{chngcntr}

\theoremstyle{definition}
\newtheorem{example}{Example}
\newcounter{thm}[section]

\theoremstyle{plain}
\newtheorem{theorem}[thm]{Theorem}
\numberwithin{thm}{section}
\newtheorem{defn}{Definition}
\newtheorem*{conj*}{Conjecture}

\newtheorem{remark}{Remark}
\newtheorem{lemma}[thm]{Lemma}
\newtheorem{prop}[thm]{Proposition}
\newtheorem{cor}[thm]{Corollary}
\newtheorem{P-C}[thm]{Theorem}
\newtheorem{H-L}[thm]{Theorem}

\begin{document}
\title[Diffusions with local time dependent noise]{Propagation of Chaos for reflecting diffusions with local-time dependent noise}
\author{Clayton Barnes}
\address{cbarnes@campus.technion.ac.il, Technion-Israel Institute of Technology}
\thanks{This research was supported by Swiss grant FNS 200021\_175728/1 and the Zuckerman STEM leadership program. The author is currently a Zuckerman Postdoctoral Scholar in the faculty of Industrial Engineering and Management at Technion-Israel Institute of Technology.}





\begin{abstract}
We prove existence and uniqueness of a reaction-diffusion equation whose diffusivity is a non-linear functional of the boundary temperature. We do this by studying systems of one-dimensional reflecting diffusions whose noise
is a function of the  reflection local time of the system, and by characterizing the large-scale (hydrodynamic) behavior by 
showing propagation of chaos. In addition, we analyze the one-particle case by computing the distribution of the hitting times of its reflection local time. 
This work is the noise analog of work done by Frank Knight (2001).
\end{abstract}
\maketitle
\tableofcontents





\section{Introduction}
\subsection{Description of model}
For a given $n$, let $(\Omega, \prob, (\mc{F}_t)_{t \geq 0})$ be a 
probability space supporting $n$ independent $\mc{F}_t-$adapted 
Brownian motions $B_i,$ for $i =1, \dots, n.$ Let $\sigma: [0, \infty) \to (0, \infty)$ be a Lipschitz function. We construct a system of continuous
$\mc{F}_t$-adapted processes $(X_i^{(n)}, L_i^{(n)})$ where the following holds
almost surely, for all $t \in [0, T]$ and $i = 1, \dots, n$:
\begin{align}\label{eq:system_intro}
	\begin{split}
		X_i^{(n)}(t) &= X_i^{(n)}(0) + L_i^{(n)}(t) + \int_0^t\sigma\big( \mb{L}^{(n)}(s)\big)\, \md B_i(s) \geq 0,\\
		\mb{L}^{(n)}(t) &:= \frac{1}{n}\sum_{i = 1}^nL_i^{(n)}(t),\\
		L_i^{(n)}(t) &\text{ is the local time of $X_i^{(n)}$ at zero at time $t.$}
	\end{split}
\end{align}
Informally, $X_1^{(n)}, \dots, X_n^{(n)}$ are one dimensional reflecting diffusions whose diffusivity depends on the collected (reflection) local time. Consequently, the $n$ diffusions are coupled together through the diffusivity function $\sigma.$ If $\sigma$ is constant, then $X_1^{(n)}, \dots, X_n^{(n)}$ are independent, but are dependent otherwise. For instance, if one of the particles has an unusually large reflection local time, this will alter the oscillations of the other particles.

We determine the limiting behavior of the empirical process 
\begin{align}\label{eq:emp_process}
	\pi^{(n)}_t = \sum_{i = 1}^n\d_{\{X_i^{(n)}(t)\}}
\end{align}
as the unique solution to the reaction-diffusion equation \eqref{intro:eq:PDE1}-\eqref{intro:eq:PDE4} by proving propagation of chaos for the system. See Theorem \ref{intro:thm:progagation_chaos}. Often, one appeals to known results in the PDE literature that guarantees uniqueness of limiting PDE. However, this does not seem to be covered in the case of our non-linear free boundary problem. In this article, we show existence and uniqueness of the PDE using a stochastic representation theorem (see Section \ref{subsec:E_U}).

\subsection{Free boundary problem}
\label{intro:PDE}
Let $f(x) \geq 0$ be a probability density supported on $[0, \infty)$. Consider the following non-linear equation for
$(q, \ell)$
\begin{flalign}
	&\frac{\partial q(t, x)}{\partial t} = \frac{1}{2}\sigma^2(\ell(t))\Delta_xq(t, x), \ x > 0\label{intro:eq:PDE1}\\
	&\frac{\partial^+ q(t, 0)}{\partial^+ x} = 0,\label{intro:eq:PDE2}\\
	&\ell(t) = \frac{1}{2}\int_0^tq(s, 0)\sigma^2(\ell(s))\,\md s,\label{intro:eq:PDE3}\\
	&\lim_{t \downarrow 0}q(t, \md x) = f(\md x). \label{intro:eq:PDE4}
\end{flalign}
We think of $q(t, x)$ as the temperature at time-space location $(t, x)$. Here \eqref{intro:eq:PDE2} is the Neumann boundary condition, which is equivalent to the boundary point at zero being a perfect insulator, so that heat is conserved in the system. Because the initial condition $u(0, x) = f(x)$ is a density, $q(t, \cdot)$ defines a probability density for any $t > 0$, and \eqref{intro:eq:PDE4} means the distribution defined by $q(t, \md x)$ converges to the distribution defined by the initial condition $f$ as $t \downarrow 0.$
The relations \eqref{intro:eq:PDE1} and \eqref{intro:eq:PDE3} imply the diffusivity of heat, $\sigma^2(\ell(t))$, is a non-linear functional of the past temperature at the insulator $q(s, 0).$ If such a solution exists, it follows automatically from \eqref{intro:eq:PDE3} that $\ell(t)$ is differentiable.
 Note that the PDE problem above is for the pair $(q, \ell)$ and is consequently non-linear because of the interaction between the diffusivity and the insulator's temperature:

\[
\text{temperature at insulator } \overset{\text{Equation \eqref{intro:eq:PDE3}}}{\longleftrightarrow} \text{ diffusivity of heat}.
\]

In proving well-posedness for the above free boundary problem we arrive at a stochastic represenation for both $q(t, x)$ and $\ell$ and give an example where this can be used to solve for $\ell$ explicitly.

\subsection{Main results}

\begin{theorem}
	There exists a unique classical solution to the free boundary problem in subsection \ref{intro:PDE}.
\end{theorem}

\begin{theorem}[Strong existence and uniqueness of particle system]\label{intro:th:system_existence}
	Let $\sigma:[0, \infty) \to (0, \infty)$ be Lipschitz, $n \in \N$, and $(\Omega, \prob, (\mc{F}_t)_{t \geq 0})$ a probability space supporting $n$ i.i.d.\
	$\mc{F}_t-$Brownian motions. There continuous $\mc{F}_t-$adapted processes $(X_1^{(n)}, \dots, X_n^{(n)}, L_1^{(n)}$ $,\dots, L_n^{(n)})$ satisfying \eqref{eq:system_intro} for $t \in [0, T]$. Such a solution is pathwise unique.
\end{theorem}

Let $\mc{P}(\R)$ be the space of probability measures on $\R$ equipped with the weak topology. Notice that the system $(X_1^{(n)}, \dots, X_n^{(n)})$ is exchangeable in the sense that the law does not depend on the labeling. For such systems, propagation of chaos is equivalent to convergence in 
law of the empirical process $\pi^{(n)}$ when it is viewed as a $\mc{P}(\R)-$valued process. 
The distributional limit of $\pi^{(n)}$ is the deterministic $\mc{P}(\R)-$valued process concentrated on the transition density of
$\wt{X}_1$ constructed in Theorem \ref{progagation_chaos}.
For fixed $t,$ $\pi^{(n)}_t = \frac{1}{n}\sum_{i = 1}^n\d_{X_i^{(n)}(t)}$ is a (random) probability measure on $\R.$ Because each $X_i^{(n)}$ is continuous, $\{\pi^{(n)}_t : t \in [0, T]\}$ is a continuous $\mc{P}(\R)-$valued stochastic process. Therefore, $\pi^{(n)}$ induces a distribution on $C([0, T], \mc{P}(\R))$ with the metric of weak-convergence placed on $\mc{P}(\R).$ The hydrodynamic limit (Theorem \ref{HL:B}) characterizes the distributional limit as $n$ approaches infinity. See Section 4.1 for more details, and M\'eleard \cite{Meleard}.
\begin{theorem}[Propagation of Chaos]\label{intro:thm:progagation_chaos}\label{progagation_chaos}
	Assume $X_i^{(n)}(0) \overset{W_1}{\lra} X_i(0)$, where $W_1$ is the Wasserstein metric and $\left(X_i(0): i \in \N\right)$ are i.i.d.\ samples of a 
	non-negative random variable independent of the Brownian motions $(B_i : i \in \N)$. Then for any 
	$k,$ $(X_1^{(n)}, \dots, X_k^{(n)})$ converges in distribution to a $k$-tuple of independent processes $(\wt{X}_1, \dots, \wt{X}_k)$, with
\begin{align}\label{th:POC:X_tilde}
	\wt{X}_i(t) = X_i(0) + \wt{L}_i(t) + \int_0^t\sigma(\ell(s))\,\mr{d}B_i(s),
\end{align}
	for $i = 1, \dots, k,$
	where $\wt{L}_i$ is the local time of $\wt{X}_i(t)$ at zero, and
\begin{align}\label{th:POC:ell}
	\ell(t) = \ex \wt{L}_1(t).
\end{align}
\end{theorem}

\begin{theorem}[Hydrodynamic Limit]\label{HL:B}
	Let $q(t, x)$ be the density of $\wt{X}_1$ given in Theorem \ref{intro:thm:progagation_chaos}. The empirical process $\pi^{(n)}$ converges 
	weakly to $\{q(t, x)\md x : t \in [0, T]\}$ in the space $C([0, T], \mc{P}(\R))$.
\end{theorem}

\subsection{Results for one particle}
In this subsection we describe the basic results for the hitting time distribution of $L$ in the $n=1$ particle case.
\begin{theorem}[Existence and uniqueness]\label{intro:th:sde_sigma_p}
	Let $\sigma : [0, \infty) \to [0, \infty)$ have support $[0, \sigma_0]$ (possibly $\sigma_0 = \infty$). Assume that $\sigma$ is locally Lipschitz on $[0, \sigma_0)$. Then there exists a unique strong solution to 
	\begin{align*}
		Y(t) &= Y(0) + L(t) + \int_0^t\sigma(L(s))\, \mr{d}B(s) \geq 0,\\
		L &\text{ is the reflection local time of $Y$ at zero,}
	\end{align*}
	for all $t \in [0, \tau_{\sigma_0}).$ Here, for any $a \in [0, \sigma_0]$
	\[
	\tau_a = \inf\{t > 0: L(t) \geq a\}
	\] is the hitting time of $L$ at level $a.$ Assuming the initial condition $Y(0)$ is independent of the Brownian motion $B$, the distribution of $\tau_a$ has Laplace transform given by
	\[
	\phi(\ld) = \ex\left[\exp\{-\sqrt{2\ld}\,Y(0)\}\right]\cdot \exp\left\{ -\sqrt{2\ld}\int_0^a\frac{\md s}{\sigma(s)}\right\}.
	\]
	Furthermore,
	\begin{align}\label{eq1:th:sde_sigma_p}
		\tau_{\sigma_0} &< \infty \text{ almost surely, if $\int_0^{\sigma_0} \frac{\md s}{\sigma(s)} < \infty$,}\\
		\tau_{\sigma_0} &= \infty \text{ almost surely, if $\int_0^{\sigma_0} \frac{\md s}{\sigma(s)} = \infty$.}
	\end{align}
	In particular, when $Y(0) = 0$, the distribution of $\tau_{\sigma_0}$ is determined by its Laplace transform
	\begin{align}\label{eq:laplace_of_tau_p}
		\phi(\ld) := \ex(e^{-\ld \tau_{\sigma_0}}) = \exp\left(-\sqrt{2\ld}\int_0^{\sigma_0}\frac{\mr{d}s}{\sigma(s)}\right).
	\end{align}
\end{theorem}

\begin{remark}\label{intro:remark}
	Strong existence and uniqueness of processes similar to \eqref{eq1:intro} reflecting in an orthant follow
	from Dupuis and Ishii's work \cite{10.2307/2244772}. To see this, consider the 
	higher dimensional process $X = (X_1, X_2)$ reflecting in $G = \R_{\geq 0} \times \R$ (in the direction (1, 1)). Let $\sigma = (\sigma_{ij})$ be degenerate in the bottom row ($\sigma_{21} = 0 = \sigma_{22}$). In this way the second component $X_2$ becomes the local time of $X_1$ and 
	one can write $X_1$ as a reflected diffusion depending on its local time $X_2$. We prove 
	Theorem \ref{intro:th:sde_sigma_p} from an approximation method which, 
	in addition, gives us tools for studying the hitting times of $L$.
\end{remark}
\begin{example}[$\sigma(x) = (1- x)^p\mathds{1}(0 \leq x \leq)$]
	In this example, the noise decreases as a power $p \in [0, \infty)$ of the local time's proximity to one. If $0 \leq p < 1$ then $\tau_1 < \infty$ almost surely and the noise will almost surely disappear, while $\tau_1 = \infty$ almost surely if $p \geq 1$.
\end{example}

\subsection{Outline} In Section \ref{sec:integ_constr} we construct a map $f \mapsto (x_f, L)$ via an approximation method for  the one-particle case ($n = 1$) of \eqref{eq:system_intro}.  That is, we construct the following integral system for any $f \in C[0, T]$, when $\sigma$ is a non-negative Lipschitz function bounded away from zero such that 
\begin{align}
	\begin{split}\label{intro:integral_equation} 
		&x_f(t) + L(t) \geq 0 \text{ for all $t \in [0, T]$},\\
		&x_f(t) = x_f(s) + \sigma(L(s))(t -s), \text{if $L$ is flat on $[s, t]$},
	\end{split}
\end{align}
where $t \mapsto L(t)$ is a nondecreasing continuous function that is flat off the set $\{s :  x_f(s) + L(s) = 0\}$. 

We construct the pair $(Y, L)$ by replacing $f$ in \eqref{intro:integral_equation} pathwise with Brownian motion. In Section \ref{sec:existence_of_sde} we prove Theorem \ref{intro:th:sde_sigma_p} by applying results obtained in Section \ref{sec:integ_constr}
that allow us to characterize the hitting times of $L.$

In Section \ref{sec:systems} we prove the propagation of chaos, Theorem \ref{intro:thm:progagation_chaos}. In general, processes that interact through their local times are hard to study because the correlation structure between the particles is difficult to analyze. In our approach we study the
large scale behavior of hitting times for the local time of the system $\mathbb{L}^{(n)}$, showing that it becomes deterministic in the limit. See Propostition \ref{prop:deter_limit}. See \cite{nadtochiy2017particle}, where the authors study systems of one-dimensional particles interacting through hitting times.

We conclude with Section \ref{subsec:E_U}, where we demonstrate existence and uniqueness for the free-boundary problem described in subsection \ref{intro:PDE}. Existence follows from a stochastic representation. That is, the solution of the PDE is the transition density of the limiting process for a fixed particle of the system, as described in the statement of propagation of chaos. Uniqueness of the PDE is demonstrated by first showing uniqueness of the diffusivity function $\ell(t)$, which is done by a coupling argument. 

\subsection{Background}
Historically, the study of macroscopic behavior  for systems of randomly interacting particles began in 1956 by Kac \cite{Kac} and continued with McKean \cite{McKean}
in 1969. This was followed by fundamental contributions during the 1980's by 
Sznitman \cite{Sznitman2,Sznitman1}, Tanaka \cite{Tanaka}, G\"artner \cite{Gartner} and many others. The hydrodynamic limit of a system of interacting particles is sometimes referred to as the macroscopic behavior of the system or the asymptotic behavior of the empirical measures, such
a result is closely related to propagation of chaos, and the two are equivalent when 
the system of interacting particles satisfies an exchangability condition \cite{Sznitman2, Sznitman1, Meleard}. For a history of hydrodynamic limits see \cite{golse2005hydrodynamic} and \cite{chen2017systems}. Systems of randomly interacting particles are probabilistic models originally motivated by 
statistical mechanics and statistical thermodynamics, particularly the theory explored 
by Maxwell, Boltzmann, and Vlasov who describe the deterministic evolution of the distribution of 
gas. A good review of interacting particle systems of the McKean-Vlasov type is found in M\'el\'eard \cite{Meleard}.

Fan \cite{fan2016discrete} shows that a
reflection term for a random walk reflecting inside a discretization of a sufficiently smooth domain will approximate the local time of reflected Brownian
 motion. We mention that our results also give a discrete approximation scheme for solutions to \eqref{eq1:intro}. 

There is a large variety of interacting particle systems giving rise to many different limiting 
behaviors. See the above mentioned works of Tanaka \cite{Tanaka}, Sznitman \cite{Sznitman1,Sznitman2}, as well as Skorohod \cite{skorohod1987stochastic}, Nadtochiy and Shkolnikov \cite{nadtochiy2017particle}, Chen and Fan \cite{chen2017systems}, Coghi et al \cite{coghi2021mckean} to mention some. For models of interacting particles with rank dependence, see Sarantsev 
\cite{sarantsev2015triple,sarantsev2017infinite}, Karatzas, Pal and Shkolnikov \cite{karatzas2016systems}, and Cabezas et al.\ who study out-of-equilibrium behavior of particles 
interacting through their ranks \cite{cabezas2019brownian}.

We briefly bring attention to the relatively recent study of stochastic free boundary problems. 
These are essentially SPDE's with a free boundary. See \cite{kim2012stochastic}.

In \cite{barnes2017hydrodynamic}, the author demonstrated hydrodynamic behavior using Lipschitz properties of Skorohod maps. We also employ Skorohod maps for constructing the solutions to \eqref{eq:system_intro}. However, the construction of the Skorohod maps, and their use in demonstrating hydrodynamic behavior is completely different than the methods we use here. We use a
stochastic representation (see Remark \ref{rem:change_time_ell_int_rep}) to prove existence and uniqueness of the free boundary problem without relying on existence and uniqueness theorems from
the theory of PDEs. This is the first existence and uniqueness result for the free boundary problem we study, as it seems not to be subsumed by known results in the analysis literature; see \cite{fasano1977generalI,fasano1977generalII,fasano1977generalIII}, for 
existence and uniqueness for generalizations of the Stephan problem.


Because hitting times for the local time process is employed in our proof of the propagation of chaos, we include a study of the $n = 1$ case for  \eqref{eq:system_intro}, which is simply a one-dimensional reflected diffusions whose noise term depends on the reflection local time.
For any filtered probability space $(\Omega, \prob, (\mc{F}_t)_{t \geq 0})$ satisfying the 
usual conditions and supporting an $\mc{F}_t-$adapted Brownian motion $B$, consider 
an $\mc{F}_t$-adapted continuous pair $(Y, L)$ such that almost surely:
\begin{align}\label{eq1:intro}
Y(t) = Y(0) + \int_0^t\sigma(L(s))\, \mr{d}B(s) + L(t) \geq 0, \text{ for all $t \in [0, T]$,}
\end{align}
where $L$ is the local time of $Y$ at zero and where $\sigma$ is non-negative and locally Lipschitz on its support. We show strong existence of such a process and characterize the hitting time distribution for $L.$ The process $Y$ has a generalized It\^o-Tanaka rule which we describe. Notice that $Y$ is not a Markov process because its noise depends on its path history. However, $(Y, L)$ is a Markov process.

Intuitively, $Y$ is a
process reflected inside $[0, \infty)$ whose noise changes upon every contact with the boundary at zero. 
For example, if $\sigma$ is a strictly decreasing function with support $[0, \sigma_0]$ (possibly with $\sigma_0 = \infty$) the noise of $Y$ will lose ``power'' upon reflecting.
Because $\sigma$ decreases to zero, it is conceivable that the noise will completely disappear at some random time
\begin{align}\label{hitting_time: tau}
\tau_{\sigma_0} = \inf\{s : \sigma(L(s)) = 0\} = \inf\{s : L(s) \geq \sigma_0\}.
\end{align}
 Whenever $\tau_{\sigma_0}$ is finite the process $Y$ will have lost all noise at time $\tau_{\sigma_0}$. That 
 is, we can continuously extend the noise $\sigma \circ L$ to be zero from time $\tau_{\sigma_0}$ onward. Hence we call $\tau_{\sigma_0}$ the \emph{time of determinacy}. We characterize $\tau_{\sigma_0}$ and show it is either a.s.\ finite or a.s.\ infinite
by characterizing the distribution of the hitting times of $L$ in general. See Theorem \ref{intro:th:sde_sigma_p}.

\subsection{Definitions}
We list the following definitions that will be used throughout the paper.
\begin{defn}
	We define the space $C[0, T]$ as the set of continuous functions $f : [0, T] \to \R$ with the metric defined by the uniform norm.
\end{defn}

\begin{defn}
	For a function $f:[0, T] \to \R$ and a set $A \subset [0, T]$, define
	\[
	\|f\|_{A} = \sup_{x \in A}|f(x)|.
	\] Similarly, we let
	\[
		\omega(f, \delta, T) := \sup_{\substack{|t - s| < \d \\ s, t \in [0, T]}}|f(t) - f(s)|
	\]
be the modulus of continuity of $f$ on $[0, T]$ depending on $\d.$
\end{defn}
\begin{defn}
	We define $(\mathcal{P}(\R), W_1)$ as the space of probability measures on $\R$ together with the Wasserstein-1 metric. In general, $\mu_n \to \mu$ under
	$W_1$ if there exists a probability space supporting random variables $X$ and $X_n, n \in \N$ such that $X \dist \mu, X_n \dist \mu_n$ for all $n \in \N,$ and $X_n \to X$ almost surely and in $L_1.$ See \cite{villani2003topics}.
\end{defn}
\begin{defn}\label{defn:tau_f}
	For $g \in C[0, T]$, define 
	\[
	M^g(t) = \min_{s \in [0, t]}(-g(s)) \lor 0
	\]
	as the signed running minimum of $g$ below zero. We define 
	\[
	\tau_g(x) = \inf\{t > 0 : M^g(t) > x\} \land S,
	\] where $S = \sup\{t : g(t) < g(T)\}$, as a right continuous inverse of $M^g$, and $\inf \emptyset = \infty.$ It is the unique right continuous inverse on its restricted domain $[0, g(T)).$
\end{defn}
\begin{defn}
Let $0 = t_0 < t_1 < \cdots < t_k = T$ be a partition of the interval 
$[0, T]$. Given another partition $0 = x_0 < x_1 < \cdots < x_m = T$, we define
the common refinement of $Q = \{t_i : i = 0, \dots, k\}$ and $P = \{x_i : i = 0, \dots, m\}$ to be the partition of $[0, T]$ formed from their union. A partition $Q$ is called finer than another partition $Q'$ if $Q' \subset Q.$
\end{defn}

\section{Construction of Skorohod map}\label{sec:integ_constr}
For a continuous function $f$ and positive Lipschitz function $\sigma$ we give a well-defined meaning to the following system:
\begin{align}
\begin{split}\label{eq:integ_eq}
&x_f(t) := f(0) + \int_0^t\sigma(L(s))\, \mr{d}f(s),\\ 
&x_f(t) + L(t) \geq 0 \text{ for all $t \in [0, T]$},
\end{split}
\end{align}
where $t \mapsto L(t)$ is a nondecreasing continuous function that is flat off the set $\{s :  = x_f(s) + L(s) = 0\}$. Because $L$ is nondecreasing and continuous it defines a measure on $[0, T].$ We show the existence and uniqueness of $(x_f, L)$ with the property that for an interval $[t^*, t)$ on which $L$ does not increase (i.e. the measure induced by $L$ gives zero measure to $[t^*, t)$), we have
\[
x_f(s) = x_f(t^*) + \sigma(L(t^*))(f(s) - f(t^*))
\]
for any $s \in [t^*, t).$ We use the integral notation in \eqref{eq:integ_eq} because the increments of $x_f$ are the increments of
$f$ scaled by $\sigma(L(t^*))$ in such an interval. Furthermore, when $f$ is replaced pathwise by a Brownian motion, the resulting process $(x_B + L, L)$ will be a strong solution to the pair $(Y, L)$ described in Theorem \ref{intro:th:sde_sigma_p}. We show this in Section \ref{sec:existence_of_sde}.


We construct $(x_f, L)$ as a limit of a sequence of approximations $(x_f^{(n)}, L^{(n)})$. For a given $n \in \N$, we define the pair $(x_f^{(n)}, L^{(n)})$ by inducting over times segments $[t_i^{(n)}, t_{i + 1}^{(n)})$, via

\reqn
\begin{align}
\begin{split}\label{def_x_f_n}
&t_0^{(n)} = 0,\\ 
&t_{i + 1}^{(n)} = \inf\left\{t > t_i^{(n)} : x_f^{(n)}(t_i^{(n)}) + \sigma\Big(\frac{i}{n}\Big)(f(t) - f(t_i^{(n)})) < -\frac{i+1}{n}\right\},\\
&x_f^{(n)}(0) = f(0), \\
&x_f^{(n)}(t) = x_f^{(n)}(t_i^{(n)}) + \sigma\Big(\frac{i}{n}\Big)(f(t) - f(t_i^{(n)})) \text{ for $t \in [t_i^{(n)}, t_{i + 1}^{(n)}]$},
L^{(n)}(t) := M^{x_f}(t).
\end{split}
\end{align}
\leqn
We use the convention that $\inf \emptyset = \infty$. We conceal $\sigma$ for convenience in the notation of $(x_f^{(n)}, L^{(n)})$ but may specifically mention $\sigma$ if particular clarification is needed.

\begin{lemma}\label{lemma:lemma1}\text{\\}
\begin{enumerate}[label = (\roman*)]
\item \[
x_f(0) = f(0) \geq 0, \ x_f(t_i^{(n)}) = -i/n \text{ for $i \in \N$},\]
\item \[
\ds f(t_{i+1}^{(n)}) - f(t_i^{(n)}) = -\frac{1}{n\sigma(i/n)}\text{ when $t_{i + 1}^{(n)} < \infty$},\]
\item \[\ds
t_i^{(n)} = \inf\Big\{t > 0: M^f(t) > \sum_{j = 0}^i\frac{1}{n\sigma(j/n)}\Big\},\]
\item For each $n,$ \[
\sup_{i \in \N}t_i^{(n)} = \infty,\]
\item \[
\ds x_f^{(n)}(t) = \sum_{i = 0}^\infty \sigma(i/n)(f(t \land t_{i + 1}^{(n)}) - f(t \land t_i^{(n)})),
\]
\item \[
\ds
\frac{\lfloor nL^{(n)}(t)\rfloor}{n} = \frac{i}{n}
\text{ for $t \in [t_i^{(n)}, t_{i + 1}^{(n)}),$}\]
\item \[
\ds x_f^{(n)}(t) - x_f^{(n)}(0) = \sum_{i = 0}^\infty\sigma\left(\frac{\lfloor nL^{(n)}\big(t_i^{(n)}\big) \rfloor}{n}\right)\big(f(t \land t_{i + 1}^{(n)}) - f(t \land t_i^{(n)})\big),\]
\item 
\[
x_f^{(n)}(t) = x_f^{(n)}(s) + \sigma(i/n)(f(t) - f(s)) \text{ for $t_i^{(n)} \leq s \leq t \leq t_{i + 1}^{(n)}$},
\]
\item For $s < t$, and $k$ with $s \in [t_k^{(n)}, t_{k + 1}^{(n)})$,
\[
x_f^{(n)}(t) = x_f^{(n)}(s) + x_g^{(n)}(t - s)
\]
where $g(x) = x_f^{(n)}(s) + L^{(n)}(s) + f(s + x) - f(s)$ and $x_g^{(n)}$ is given by \eqref{def_x_f_n} using $\sigma_g(x) = \sigma\Big(\frac{k}{n} + x\Big),$
\item $(x_f^{(n)}(t), L^{(n)}(t))$ depends on $f|_{[0, t]}.$ For $s < t$, $(x_f^{(n)}(t), L^{(n)}(t))$ is a function of $(x_f^{(n)}(s), L^{(n)}(s))$ and $f|_{[s, t]}.$
\end{enumerate}
\end{lemma}
\begin{proof}
	(i)-(iii), (viii), (ix), and (x) follow from the definition, (iv) follows from the Lipschitz property of $\sigma,$ (v)-(vii) follow by an induction argument. 
\end{proof}

We will use the following results to show existence of subsequential limits of $(x_f^{(n)}, L^{(n)})$ in $C[0, \tau_f(1/K - \e)]$ for $\e \in (0, 1/K)$. Recall $\tau_f$ is given in Definition \ref{defn:tau_f}.
\begin{lemma}\label{lemma2}
Let $K$ be the Lipschitz constant for $\sigma.$ For any fixed $\e > 0,$ there is a $C(\e) \in \N$ such that 
\[
\sum_{i = 0}^{C(\e)n}\frac{1}{n\sigma(i/n)} > \frac{1}{K} - \e,
\]
for all $n$.
\end{lemma}
\begin{proof}
Because $\sigma(i/n) \leq \sigma(0) + Ki/n$, we have
\begin{align*}
&\sum_{i = 0}^{Cn}\frac{1}{n\sigma(i/n)} \geq \sum_{i = 0}^{Cn}\frac{1}{n(\sigma(0) + Ki/n)}\\
&\geq \sum_{i = 0}^{Cn}\frac{1}{n\sigma(0) + Ki}\geq \sum_{i = 0}^{Cn}\frac{1}{n\sigma(0) + KCn}\\
&\geq \frac{Cn}{n(KC + \sigma(0))} = \frac{C}{KC + \sigma(0)},
\end{align*}
for any $C \in \N.$ This last lower bound approaches $1/K$ as $C \to \infty$. Hence for any $\e > 0$ there is a $C(\e) \in \N$ with 
\[
\sum_{i = 0}^{Cn}\frac{1}{n\sigma(i/n)} \geq \frac{C(\e)}{KC(\e) + \sigma(0)} > \frac{1}{K} - \e.
\]
\end{proof}

\begin{cor}\label{cor:bound_sigma}
\[
\sup_{t \in [0, \tau_f(1/K - \e)]}\sigma\Big(\frac{\lfloor nL^{(n)}(t) \rfloor}{n}\Big) \leq \frac{\sigma(0)}{n} + KC(\e),
\]
where $C(\e)$ is given from the Lemma \ref{lemma2}.
\end{cor}
\begin{proof}
It follows from Lemma \ref{lemma:lemma1} (iii) and (vi), and Lemma \ref{lemma2}, that $t_{nC(\e)}^{(n)} \geq \tau_f(1/K - \e)$. Then,
\[
\sup_{t \in [0, \tau_f(1/K - \e)]}\sigma\Big(\frac{\lfloor nL^{(n)}(t) \rfloor}{n}\Big) \leq \sup_{i = 0, \dots, C(\e)n}\sigma\Big(\frac{i}{n}\Big) \leq \frac{\sigma(0) + KC(\e)n}{n}.
\]
\end{proof}

\noindent
By Lemma \ref{lemma:lemma1} (v) it is reasonable to think the oscillations of $x_f^{(n)}$ can be controlled
since we have bounds on $\sigma\Big(\frac{\lfloor nL^{(n)}(t) \rfloor}{n}\Big)$.
Indeed, the following proposition gives uniform control over the oscillations of $x_f^{(n)}$ in terms of the oscillations of $f$ in the interval $[0, \tau_f(1/K - \e)].$
\begin{prop}\label{prop:prop_x_f_bound}
For any $0 \leq s < t \leq \tau_f(1/K - \e)$, there is a constant $C' = C'(\e, K, \d)$ such that
\begin{align*}
&|x_f^{(n)}(t) - x_f^{(n)}(s)|\\
&\leq 4K\|f(t)\|_{[0, \tau_f(1/K - \e)]}\Big(\frac{1}{n} + C'\sup_{x \in [s, t]}|f(x) - f(s)|\Big) + C'\sup_{x \in [s, t]}|f(x) - f(s)|.
\end{align*}
Here $C'(\e, K, \d) = 2(\sigma(0) + KC(\e)) + K + 2 + 1/\d$, where $C(\e)$ given in Lemma \ref{lemma2}, and $\d > 0$ is the
uniform lower bound of $\sigma.$
\end{prop}
\begin{proof}
We use the representation (vii) in Lemma \ref{lemma:lemma1} to compute
\begin{align}
\begin{split}\label{eq2}
&|x_f^{(n)}(t) - x_f^{(n)}(s)| =\\
&\Big|\sigma(k_{n-1}/n)(f(t_{k_{n - 1}}^{(n)}) - f(s)) + \sigma(m_n/n)(f(t) - f(t_{m_n}^{(n)})\\
&+ \sum_{i = k_n}^{m_n}\sigma\Big(\frac{\lfloor nL^{(n)}\big(t_i^{(n)}\big) \rfloor}{n}\Big)\big(f(t_{i + 1}^{(n)}) - f(t_i^{(n)})\big)\Big|\\
&\leq |\sigma(k_{n-1}/n)(f(t_{k_{n - 1}}^{(n)}) - f(s)) + \sigma(m_n/n)(f(t) - f(t_{m_n}^{(n)})| \\
&+ \sum_{i = k_n}^{m_n}\Big|\sigma\Big(\frac{\lfloor nL^{(n)}\big(t_i^{(n)}\big) \rfloor}{n}\Big)\big(f(t_{i + 1}^{(n)}) - f(t_i^{(n)})\big)\Big|,
\end{split}
\end{align}
where $k_n = \min\{i : t_i^{(n)} > s\}$ is the index of the first element in the partition $\{t_i^{(n)}\}$ occurring after $s$, and similarly
$m_n = \max\{i : t_i^{(n)} < t\}$ is the index of the last element occurring before $t.$ To control the first term of the last inequality above, we add and subtract $\sigma(k_{n - 1}/n)f(t) + \sigma(m_n/n)f(t_{k_{n - 1}}^{(n)})$
 to rewrite
\[
\sigma(k_{n-1}/n)(f(t_{k_{n - 1}}^{(n)}) - f(s)) + \sigma(m_n/n)(f(t) - f(t_{m_n}^{(n)})
\]
as
\begin{align*}
&f(t)\big(\sigma(m_n/n) - \sigma(k_{n - 1}/n)\big) + \sigma(k_{n - 1}/n)\big(f(t) - f(s)\big)\\
&+ f(t_{k_{n - 1}}^{(n)})\big(\sigma(k_{n - 1}/n) - \sigma(m_n/n)\big) + \sigma(m_n/n)\big(f(t_{k_{n - 1}}^{(n)}) - f(t_{m_{n}}^{(n)})\big)
\end{align*}
Hence
\begin{align}
\begin{split}\label{eq:eqbound}
&|\sigma(k_{n-1}/n)(f(t_{k_{n - 1}}^{(n)}) - f(s)) + \sigma(m_n/n)(f(t) - f(t_{m_n}^{(n)})|\\
&\leq \big(|f(t)| + |f(t_{k_{n - 1}}^{(n)})|\big)|\sigma(m_n/n) - \sigma(k_{n - 1}/n)|\\
&+ \big(\sigma(m_n/n) + \sigma(k_{n-1}/n)\big)\big(|f(t) - f(s)| + |f(t_{k_{n-1}}^{(n)}) - f(t_{m_n}^{(n)})| \big)\\
&\leq 2\|f\|_{[0, \tau_f(1/K - \e)]}|\sigma(m_n/n) - \sigma(k_{n - 1}/n)| \\
&+ 2(\sigma(0) + C(\e)K)\big(|f(t) - f(s)| + M^f(t_{m_n}^{(n)}) - M^f(t_{k_{n-1}}^{(n)})\big), \text{ from Lemma \ref{lemma2}}\\
&\leq 2\|f\|_{[0, \tau_f(1/K - \e)]}K\Big(\frac{m_n - k_{n - 1}}{n}\Big) + 2(\sigma(0) + C(\e)K)\Big(|f(t) - f(s)| + \sum_{j = k_{n - 1}}^{m_n}\frac{1}{n\sigma(j/n)}\Big)\\
&\leq 2\|f\|_{[0, \tau_f(1/K - \e)]}K\Big(\frac{m_n - k_{n - 1}}{n}\Big) + 2(\sigma(0) + C(\e)K)\Big(|f(t) - f(s)| + \frac{m_n - k_{n - 1}}{\d n}\Big),
\end{split}
\end{align}
where we recall that $\d > 0$ is the uniform lower bound on $\sigma.$
By definition $m_n - k_{n - 1}$ is the number of elements in the partition 
$\{t_i^{(n)}\}$ containing the interval $[s, t].$ 
From Lemma \ref{lemma:lemma1} (iii), $t_{i+ 1}^{(n)} - t_i^{(n)}$ is the time
it takes $M^f(t_i^{(n)} + s)$ to increase by $1/(n\sigma(i/n)).$ By Corollary \ref{cor:bound_sigma}, $1/(n\sigma(i/n)) \geq 1/(\sigma(0) + nKC(\e))$. Hence $t_{i + 1}^{(n)} - t_i^{(n)}$ is no less than the time taken by $M^f(t_i^{(n)} + s)$ to increase by $1/(\sigma(0) + nKC(\e)).$ Therefore the total number of the partition times $\{t_i^{(n)}\}$ contained in $[s, t]$ is no 
more than $M^f(t) - M^f(s)$ divided by this gap $1/(\sigma(0) + nKC(\e)).$ Thus
\(
m_n - k_{n - 1} \leq 2 + (\sigma(0) + nKC(\e))(M^f(t) - M^f(s)).
\)
The addition of 2 comes by counting the first and last intervals
$[t_{k_{n - 1}}^{(n)}, t_{k_n}^{(n)})$ and $[t_{m_{n + 1}}^{(n)}, t_{m_n}^{(n)}).$ Continuing from \eqref{eq:eqbound}, let $C'(\e) = \sigma(0) + KC(\e),$
\begin{align}
\begin{split}\label{eq:bound_1st_term}
&|\sigma(k_{n-1}/n)(f(t_{k_{n - 1}}^{(n)}) - f(s)) + \sigma(m_n/n)(f(t) - f(t_{m_n}^{(n)})|\\
&\leq 2\|f\|_{[0, \tau_f(1/K - \e)]}K\Big(\frac{2 + (\sigma(0) + nKC(\e))(M^f(t) - M^f(s))}{n}\Big) \\
&+ 2(\sigma(0) + C(\e)K)\Big(|f(t) - f(s)| + \frac{2 + (\sigma(0) + nKC(\e))(M^f(t) - M^f(s))}{\d n} \Big)\\
&\leq 2\|f\|_{[0, \tau_f(1/K - \e)]}(1 +K)\Big(\frac{2}{n} + C'(\e)(M^f(t) - M^f(s))\Big)\\
&+ 2C'(\e)|f(t) - f(s)| + 2C'(\e)\frac{M^f(t) - M^f(s)}{\d}\\
&\leq \frac{4K\|f\|_{[0. \tau_f(1/K - \e)]}}{n} + 2C'(\e)(\|f\|_{[0, \tau_f(1/K - \e)]} +
K + 1 + 1/\d)\sup_{x \in [s, t]}|f(x) - f(s)|.
\end{split}
\end{align}


We bound the sum in equation \eqref{eq2}, 
\begin{align}
\begin{split}\label{bound3}
&\sum_{i = k_n}^{m_n}\Big|\sigma\Big(\frac{\lfloor nL^{(n)}\big(t_i^{(n)}\big) \rfloor}{n}\Big)\big(f(t_{i + 1}^{(n)}) - f(t_i^{(n)})\big)\Big|\\
&\leq \sum_{i = k_n}^{m_n}C'(\e)|f(t_{i + 1}^{(n)}) - f(t_i^{(n)})|, \text{ by Corollary \ref{cor:bound_sigma}},\\
&\leq \sum_{i = k_n}^{m_n}C'(\e)\big(M^f(t_{i + 1}^{(n)}) - M^f(t_i^{(n)})\big)\\
&= C'(\e)\big(M^f(t_{m_n}^{(n)}) - M^f(t_{k_n}^{(n)})\big)\\
&\leq C'(\e)\big(M^f(t) - M^f(s)\big)\\
&\leq C'(\e)|f(t) - f(s)|,
\end{split}
\end{align}
since $[t_{k_n}^{(n)}, t_{m_n}^{(n)}] \subset [s, t]$ be definition of $k_n, m_n.$ 

Combining bounds \eqref{eq:bound_1st_term} and \eqref{bound3}, \eqref{eq2} becomes
\begin{align*}
|x_f(t) - x_f(s)| &\leq \frac{4K\|f\|_{[0, \tau_f(1/K - \e)]}}{n} \\
&+ 2C'(\e)(\|f\|_{[0, \tau_f(1/K - \e)]}+K + 2 + 1/\d)\sup_{x \in [s, t]}|f(x) - f(s)|.
\end{align*}
\end{proof}
\begin{cor}\label{cor:tightness_x_f_n}
The collection of functions $\{(x_f^{(n)}(s), L^{(n)}(s)) : n \in \N\}$ is tight
in the space $C[0, \tau_f(1/K) - \e]$ with the uniform norm.
\end{cor}
\begin{proof}
By definition, $L^{(n)} = M^{x_f}$, so
it suffices to show that $\{x_f^{(n)} : n \in \N\}$ is tight. It follows directly from Proposition \ref{prop:prop_x_f_bound}, and $x_f^{(n)}(0) = f(0),$ that 
$\{x_f^{(n)} : n \in \N\} \subset C[0, \tau_f(1/K - \e)]$ satisfies the equicontinuity and uniform boundedness criteria of the Arzel\`a-Ascoli theorem.
\end{proof}

\begin{prop}\label{prop:x_f_refinement}
Let $P = \{0 = t_0^{(n)} < t_1^{(n)} \land \tau_f(1/K - \e) < \cdots < t_{nC(\e)}^{(n)} \land \tau_f(1/K - \e) = \tau_f(1/K - \e)\}$ be a partition of $[0, \tau_f(1/K - \e)]$. Let $Q = \{x_i : i = 0, \dots, m\}$ be another partition finer than $P$. Then 
\begin{align}\label{eq:x_f_refinement}
x_f^{(n)}(t) -x_f^{(n)}(0) = \sum_{i = 0}^m\sigma\Big( \frac{\lfloor nL^{(n)}(x_i) \rfloor}{n}\Big)\big(f(t \land x_{i + 1}) - f(t \land x_i)\big)
\end{align}
for any $t \in [0, \tau_f(1/K - \e)].$
\end{prop}
\begin{proof}
This follows from Lemma \ref{lemma:lemma1} (vii) and (viii).
\end{proof}
\begin{cor}\label{cor:dist_x_n-x_m}
For any $m,n \in \N$, and $t \in [0, \tau_f(1/K - \e)]$, we have
\[
|x_f^{(n)}(t) - x_f^{(m)}(t)| \leq \sum_{i = 0}^\infty K\big(|L^{(n)}(x_i) - L^{(m)}(x_i)| + 2/(m \lor n)\big)\big|f(t \land x_{i + 1}) - f(t \land x_i)\big|,
\]
where $\{x_i\}$ is the partition of $[0, \tau_f(1/K - \e)]$ formed from the common refinement of $\{t_i^{(n)}\}$ and $\{t_i^{(m)}\}.$
\end{cor}
\begin{proof}
Apply Proposition \ref{prop:x_f_refinement} and the Lipschitz property of $\sigma.$
\end{proof}
\begin{lemma}\label{lemma:closure_times}
Fix $\e > 0$, and let $x_f^{(n_k)}$ be a convergent subsequence in $C[0, \tau_f(1/K - \e)]$, as 
guaranteed by Corollary \ref{cor:tightness_x_f_n}, and denote the limit as $x_f.$ Fix $\tau = \tau_f(1/K - \e)$ for a given $\e > 0$. Let \[
E_f = \{\tau_f(a) : a > 0\} \cap [0, \tau],
\] and denote 
\begin{align*}
A_{n_k} &= \big\{t_i^{(n_k)} : i = 0, \dots, nC(\e)\big\} \cap [0, \tau],\\
A &= \bigcup_{n \in \N}A_{n_k}.
\end{align*} Then
\begin{align}\label{eq:lemma:closure_times}
\overline{E}_{x_f} = \overline{A} = \overline{E_f},
\end{align}
where $\overline{E}$ denotes the closure of $E.$ Furthermore, $L^{(n_k)}$ converges to $M^{x_f}$ in $C[0, \tau]$, and we have
\begin{align}\label{eq2:lemma:closure_times}
L(x) := M^{x_f}(x) = -x_f(x)
\end{align}
for all $x \in \overline{E}_f.$ For any $t \in [0, \tau]\setminus \overline{E}$,
\begin{align}\label{eq:shift:lemma:closure_times}
x_f(t) = x_f(t^*) + \sigma(L(t^*))\big(f(t) - f(t^*)\big),
\end{align}
where $t^* = \max\{s < t : s \in \overline{E}\}$.
\end{lemma}
\begin{proof}
By Lemma \ref{lemma:lemma1} (iii), $A_n \subset E_f$. Hence $A \subset E_f$ and
$\overline{A} \subset \overline{E}_f.$ We assume there exists a $\delta$ such that $\sigma > \d > 0$, so
Lemma \ref{lemma:lemma1} (iii) guarantees that 
\begin{align}\label{eq1:lemma:closure_times}
\frac{1}{n\sigma(i/n)} < \frac{1}{n\d}.
\end{align}
From Lemma \ref{lemma:lemma1} (iii),
\[
A_n = \Big\{\tau_f(a_i^{(n)}) : a_i^{(n)} = \sum_{j = 0}^i\frac{1}{n\sigma(j/n)}\Big\} \cap [0, \tau].
\]
By \eqref{eq1:lemma:closure_times}, 
\[
0 < a_{i + 1}^{(n)} - a_i^{(n)} < \frac{1}{n\d}.
\]
That is, $\{a_i^{(n)} : i = 0, \dots, nC(\e)\}$ has a mesh size decreasing to zero. Consequently, for any $a \in [0, 1/K - \e]$ there is a decreasing sequence $a_{i_n}^{(n)}$ converging to $a$,  and
$\tau_f\big(a_{i_n}^{(n)}\big) \to \tau_f(a)$ by right continuity. This implies
\(
E_f \subset \overline{A},
\)
so $\overline{E}_f \subset \overline{A}$ and we have shown $\overline{E}_f = \overline{A}.$ Similarly $E_{x_f} \subset \overline{A}$.

To show \eqref{eq2:lemma:closure_times}, let $x \in \overline{E}_f$ and choose $a_{i_k}^{(n_k)} \in A_{n_k}$ such that $a_{i_k}^{(n_k)} \to x.$ We know
\[
L^{(n_k)}(a_{i_k}^{(n_k)}) = M^{x_f^{(n_k)}}(a_{i_k}^{(n_k)}) = -x_f^{(n_k)}(a_{i_k}^{(n_k)}).
\]
Taking limits as $k \to \infty$ on both sides and using the assumption that $x_f^{(n_k)}$ converges uniformly to $x_f$, we see
\[
L(x) = M^{x_f}(x) = -x_f(x).
\]
Equation \eqref{eq:shift:lemma:closure_times} follows from the convergence
of $x_f^{(n_k)}$ to $x_f$, Lemma \ref{lemma:lemma1} (vii), and the fact that
there is a sequence of times $t_{i_k}^{(n_k)}$ such that $t_{i_k}^{(n_k)} \to t^*$ with $t \in [t_{i_k}^{(n_k)}, t_{i_k + 1}^{(n_k)}).$
\end{proof}
The following result is classical and we state it as a lemma.
\begin{lemma}\label{lemma:right_continuous_classic}
Let $g \in C[0, T]$ be a nondecreasing function. Then 
\[
\tau_g(a) = \inf\{t > 0 : g(t) > a\}
\]
is the unique right continuous inverse function of $g$. That is, $\tau_g:[0, \max_x g(x)) \to \R$ is the unique right continuous map such that \(
g \circ \tau_g = id.
\)
\end{lemma}

We now prove uniqueness of the (subsequential) limits of $(x_f^{(n)}, L^{(n)}).$
\begin{theorem}\label{theorem1}
For any $\e > 0$, the sequence $(x_f^{(n)}, L^{(n)})$ converges uniformly on $[0, \tau_f(1/K - \e)]$ to a unique pair of continuous functions $(x_f, L).$ 
Where for $s, t \in [0, \tau_f(1/K - \e)]$, $s < t$,
\begin{align}\label{eq:time_shift_limit}
x_f(t) = x_f(s) + x_g(t - s)
\end{align}
for $g(x) = x_f(s) + L(s) + f(s + x) - f(s)$ under $\sigma_g(x) = \sigma(L(s) + x).$
\end{theorem}
\begin{proof}
\emph{Step 1:} Again, it suffices to show $x_f^{(n)}$ converges to a unique function $x_f$ because, in this case, $L^{(n)} := M^{x_f^{(n)}}$ will converge to $M^{x_f}.$
Let $n_k, m_k$ be two sequences such that
\begin{align}\label{eq3}
\begin{split}
&x_f^{(n_k)} \lra x_f, \text{ and }\\
&x_f^{(m_k)} \lra \wt{x}_f,
\end{split}
\end{align}
uniformly in $C[0, \tau_f(1/K - \e)].$ We will show $x_f = \wt{x}_f.$ 

\emph{Step 2:} To do 
this we first show $x_f(t) = \wt{x}_f(t)$ for all $t \in \overline{E}_f$, 
where $E_f$ is given in Lemma \ref{lemma:closure_times}. Assuming this fact,
note that $\tau_f(a) \in E_f$ for any
$a \in (0, 1/K - \e)$, \eqref{eq2:lemma:closure_times} shows $-x_f(\tau_{x_f}(a)) = -\tilde{x}_f(\tau_{x_f}(a))$. Then,
\[
M^{x_f}(\tau_{x_f}(a)) = -x_f(\tau_{x_f}(a)) = -\wt{x}_f(\tau_{x_f}(a)) = M^{\wt{x}_f}(\tau_{x_f}(a)).
\]
This means that $\tau_{x_f}$ is a right continuous inverse of $M^{\wt{x}_f}$. By Lemma \ref{lemma:right_continuous_classic}, $\tau_{x_f} = \tau_{\wt{x}_f}.$
In other words, $M^{x_f}$ and $M^{\wt{x}_f}$ are two continuous nondecreasing functions (with domain $[0, \tau_f(1/K - \e)]$) having the same right continuous inverse. Hence $M^{x_f} = M^{\wt{x}_f}$ on $[0, \tau_f(1/K - \e)].$ Then given any $t \in [0, \tau_f(1/K - \e)]$, let $t^* = \max\{s < t : s \in \ol{E}_f\}.$ From \eqref{eq:shift:lemma:closure_times} it follows that
\begin{align*}
x_f(t) &= x_f(t^*) + M^{x_f}(t^*)(f(t) - f(t^*))\\
&= \wt{x}_f(t^*) + M^{\wt{x}_f}(t^*)(f(t) - f(t^*))\\
&= \wt{x}_f(t),
\end{align*}
so that $x_f = \wt{x}_f$ on the entire interval.

\emph{Step 3:} Now we prove $x_f = \wt{x}_f$ on $\ol{E}_f.$
Denote $\{x_i\}$ as 
the partition of $[0, \tau_f(1/K - \e)]$ formed from $\big\{t_i^{(n_k)}\big\} \cup \big\{t_i^{(m_k)}\big\}.$ From Corollary \ref{cor:dist_x_n-x_m} and 
Lemma \ref{lemma:closure_times}, we have
\begin{align}\label{eq4}
\begin{split}
&|x_f^{(n_k)}(t) - x_f^{(m_k)}(t)|\\
&\leq \sum_{i = 0}^\infty K\Big(|L^{(n_k)}(x_i) - L^{(m_k)}(x_i)| + \frac{2}{n_k \lor m_k}\Big)\big|f(t \land x_{i + 1}) - f(t \land x_i)\big| \\
&\leq \sum_{i = 0}^\infty K\big(\|x_f^{(n_k)} - x_f^{(m_k)}\|_{\overline{E}_f \cap [0, t]} + \frac{2}{n_k \lor m_k}\big)\big|f(t \land x_{i + 1}) - f(t \land x_i)\big|, \text{ since $x_i \in E$,}\\
&\leq K\big(\|x_f^{(n_k)} - x_f^{(m_k)}\|_{\overline{E}_f \cap [0, t]} + \frac{2}{n_k \lor m_k}\big)\sum_{i = 0}^\infty \big(M^f(t \land x_{i + 1}) - M^f(t \land x_i)\big),
\end{split}
\end{align}
since $s_1, s_2 \in \ol{E}_f$ implies $|f(s_1) - f(s_2)| = |M^f(s_1) - M^f(s_2)|$ by Lemma \ref{lemma:closure_times}. Because $M^f$ is nondecreasing,
\[
\sum_{i = 0}^\infty \big(M^f(t \land x_{i + 1}) - M^f(t \land x_i)\big) = M^f(t).
\]
Hence,
\begin{align}\label{eq5}
\begin{split}
|x_f^{(n_k)}(t) - x_f^{(m_k)}(t)| &\leq K\Big(\|x_f^{(n_k)} - x_f^{(m_k)}\|_{\overline{E}_f \cap [0, t]} + 2/(n_k \lor m_k)\Big)M^f(t).
\end{split}
\end{align}
The right hand side of inequality \eqref{eq5} is nondecreasing, so the same upper bound holds when taking the supremum of the term on the left hand over $\overline{E}_f \cap [0, t]$:
\begin{align}\label{eq6}
\begin{split}
\big\|x_f^{(n_k)} - x_f^{(m_k)}\big\|_{\overline{E}_f \cap [0, t]} &:= \sup_{s \in \overline{E}_f \cap [0, t]}|x_f^{(n_k)}(t) - x_f^{(m_k)}(t)|\\
&\leq K\left(\big\|x_f^{(n_k)} - x_f^{(m_k)}\big\|_{\overline{E}_f \cap [0, t]} + \frac{2}{n_k \lor m_k}\right)M^f(t).
\end{split}
\end{align}
Now take $n_k, m_k \to \infty$ on both sides and apply \eqref{eq3} to see
\begin{align}\label{eq7}
\begin{split}
\big\|x_f - \wt{x}_f\big\|_{\overline{E}_f \cap [0, t]}
&\leq K\big\|x_f - \wt{x}_f\big\|_{\overline{E}_f \cap [0, t]}M^f(t).
\end{split}
\end{align}
But, in fact, $KM^f(t) < 1$ for $t \in [0, \tau_f(1/K - \e)]$ by definition
of $\tau_f.$ Consequently, unless $\|x_f - \wt{x}_f\|_{\overline{E}_f \cap [0, t]} = 0$
\[
\big\|x_f - \wt{x}_f\big\|_{\overline{E}_f \cap [0, t]} < \big\|x_f - \wt{x}_f\big\|_{\overline{E}_f \cap [0, t]}, \text{ for all $t \in \overline{E}_f$,}
\]
which is a contradiction. Hence $\|x_f - \wt{x}_f\|_{\overline{E}_f \cap [0, t]} = 0$, i.e.\ $x_f = \wt{x}_f$ on $\overline{E}_f.$ 
\end{proof}

This next proposition shows the map $f \mapsto x_f^{(n)}$ is 
asymptotically continuous with shifts of the initial condition and domain shifts
of $\sigma.$
\begin{prop}\label{prop:initial_conditions}
Let $\sigma_n(x) = \sigma(\zeta_n + x)$, where $\zeta_n,z_n \geq 0$ and $z_n \to z,\, \zeta_n \lra \zeta$. Let $x_{z + f}^\sigma$ be the unique function guaranteed by Theorem \ref{theorem1} under the function $\sigma$, and $x_{z_n + f}^{\sigma_n}$ be the unique function under $\sigma_n$. Then
\[
\big\|x_{z_n + f}^{\sigma_n} - x_{z + f}^\sigma\big\|_{[0, \tau_{z+f}(1/K - \e)]} \lra 0.
\]
\end{prop}
\begin{proof}
Proposition \ref{prop:prop_x_f_bound} can be modified to show that
the collection $x_{z_n + f}^{\sigma_n}$, constructed with $\sigma_n$, is tight in
$C[0, \tau_f(1/K - \e)]$ as in Corollary \ref{cor:tightness_x_f_n}. Take
a subsequence $n_k$ such that $x_{z_{n_k} + f}^{\sigma_{n_k}}$ converges to a function $\wt{x}$ uniformly. It suffices to show $x_f = \wt{x}.$
Since $z_n + f$ is a shift of $z + f$, $M^{z_n + f} = \big[M^{z + f} - (z_n - z)\big] \lor 0$ so $t_{i, z_n + f}^{(n)}$ are hitting times for $M^f$ by Lemma \ref{lemma:lemma1} (iii). One can show the results of Lemma \ref{lemma:closure_times} hold with 
the times $t_{i, z_{n_k} + f}^{(n_k)}$ and the set $\overline{E}_f.$ Recall that
\begin{align}\label{prop:eq1}
|f(x_{i + 1}) - f(x_i)| = M^f(x_{i + 1}) - M^f(x_i)
\end{align}
for $x_i, x_{i+1} \in \ol{E}_f.$ We incorporate the functions $\sigma$ and $\sigma_n$ in the notation of the functions $x_f^{(n)}$ given in \eqref{def_x_f_n} by letting $(x_{z_n + f}^{(n), \sigma}, L^{(n), \sigma})$ denote $(x_g^{(n)}, L^{(n)})$ with $g = z_n + f$ under $\sigma$.
We know $x_{z_n + f}^{(n), \sigma_n}$ has a form similar to that given in \eqref{eq:x_f_refinement} with the partition of $[0, \tau_f(1/K - \e)]$ given by $\{x_i\} = \{t_{i, z_n + f}^{(n), \sigma_n}\} \cup \{t_{i, z + f}^{(n), \sigma_n}\}$. That is,
\[
x_{z_n + f}^{(n), \sigma_n}(t) - (z_n + f(0)) = \sum_{i = 0}^{\infty}\sigma_n\left(\frac{\lfloor nL^{(n)}_{z_n + f}(x_i)\rfloor}{n}\right)\big(f(t \land x_{i + 1}) - f(t \land x_i)\big).
\]
We have a similar representation for $x_{z + f}^{(n), \sigma}(t) - (z + f(0)).$ By subtracting these two representations and \eqref{prop:eq1}, we have
\begin{align*}
&|x_{n_k + f}^{(n_k), \sigma_{n_k}}(t) - x_{z +f}^{(n_k), \sigma}(t)|\\
&\leq |z_{n_k} - z| + \sum_{i = 0}^{\infty}\Big|\sigma\Big(\zeta_{n_k} + \frac{\lfloor nL^{(n_k), \sigma_{n_k}}_{z_{n_k} + f}(x_i)\rfloor}{n_k}\Big) - \sigma\Big(\zeta + \frac{\lfloor nL^{(n_k), \sigma}_{z_{n_k} + f}(x_i)\rfloor}{n_k}\Big)\Big|\big|f(t \land x_{i + 1}) - f(t \land x_i)\big|\\
&\leq |z_{n_k} - z| + \sum_{i = 0}^{\infty}\Big|\sigma\Big(\zeta_n + \frac{\lfloor nL^{(n_k), \sigma_{n_k}}_{z_{n_k} + f}(x_i)\rfloor}{n_k}\Big) - \sigma\Big(\zeta + \frac{\lfloor nL^{(n_k), \sigma}_{z_{n_k} + f}(x_i)\rfloor}{n_k}\Big)\Big|\big|f(t \land x_{i + 1}) - f(t \land x_i)\big|\\
&\leq |z_{n_k} - z| + \sum_{i = 0}^{\infty}K\Big|\zeta_n + \frac{\lfloor nL^{(n_k), \sigma_{n_k}}_{z_{n_k} + f}(x_i)\rfloor}{n_k} - \zeta - \frac{\lfloor nL^{(n_k), \sigma}_{z_{n_k} + f}(x_i)\rfloor}{n_k}\Big|\big|f(t \land x_{i + 1}) - f(t \land x_i)\big|\\
&\leq |z_{n_k} - z| + K\sum_{i = 0}^\infty\Big(|\zeta_n - \zeta| + \|L_{z_{n_k} + f}^{(n),\sigma_{n_k}} - L_{z + f}^{(n), \sigma}\|_{\overline{E}_f \cap [0, t]} + \frac{2}{n_k}\Big)\big|f(t \land x_{i + 1}) - f(t \land x_i)\big|.
\end{align*}
One can then apply the same arguments used in the proof of Theorem \ref{theorem1} by first showing $x_f = \wt{x}$ on $\overline{E}_f$, then the entire interval $[0, \tau_f(1/K - \e)].$
\end{proof}
Thus far we have constructed $(x_f, L)$ by showing $(x_f^{(n)}, L^{(n)})$ converges uniformly on $[0, \tau_f(1/K - \e)]$ to a unique function. We extend this construction by showing 
that for each $T < \infty$, $(x_f^{(n)}, L^{(n)})$ converges uniformly on $[0, T]$ to a unique pair $(x_f, L)$. This extends
Theorem \ref{theorem1} to hold for $C[0, T]$ for any positive $T.$
\begin{theorem}\label{theorem:convergence}
For any $T > 0$, $(x_f^{(n)}, L^{(n)})$ converges uniformly to a unique $(x_f, L) \in C([0, T], \R^2)$. Furthermore,
\begin{enumerate}[label = (\roman*)]
\item For $t \in [0, T]$, let $t' = \sup\{0 \leq s \leq t: x_f(s) + L(s) = 0\}$ be the last zero of $x_f + L$ before $t.$ Then
\[
x_f(s) = x_f(t') + \sigma\big(L(t')\big)\big(f(s) - f(t')\big),
\]
for all $s \in [t', t].$
\item For all $0 \leq s \leq t \leq T,$
\[
x_f(t) = x_f(s) + x_g(t - s)
\]
where $g(x) = x_f(s) + L(s) + f(s + x) - f(s)$ under $\sigma_g(x) = \sigma(L(s) + x).$
\end{enumerate} 
\end{theorem}
\begin{proof}
We use a ``patching'' argument with the shifting property Lemma \ref{lemma:lemma1} (ix) to show $x_f^{(n)}$ converges uniformly on larger intervals.

Set $\tau = \tau_f(1/K - \e)$ and let $x_f$ be the limit of $x_f^{(n)}$ on $[0, \tau]$.
By Lemma \ref{lemma:lemma1} (ix), we write
\begin{align}\label{eq1:theorem:convergence}
x_f^{(n)}(t + \tau) = x_f^{(n)}(\tau) + x_{g^{(n)}}^{(n)}(t)
\end{align}
 for $t \in [0, \tau_g(1/K - \e)]$ for $g^{(n)}(x) = x_f^{(n)}(\tau) + L^{(n)}(\tau) + f(\tau + x) - f(\tau)$, $\sigma_{g^{(n)}}^{(n)}(x) = \sigma(L^{(n)}(\tau) + x),$ $g(x) = x_f(\tau) + L(\tau) + f(\tau + x) - f(\tau)$, $\sigma_g(x) = \sigma(L(\tau) + x)$, and $\e \in (0, 1/K).$ Applying Proposition \ref{prop:initial_conditions} with $\zeta_n = L^{(n)}(\tau), z_n = x_f^{(n)}(\tau) + L^{(n)}, z = x_f(\tau) + L(\tau)$ shows
\[
x_{g^{(n)}}^{(n)} \lra x_g,
\]
uniformly on $[0, \tau_g(1/K - \e)]$. Consequently, we can apply Theorem \ref{theorem1} to $x_f^{(n)}$ and $x_{g^{(n)}}^{(n)}$ with equation \eqref{eq1:theorem:convergence}. Because both 
terms on the right hand side converge uniformly, the left hand side also converges uniformly. Thus,
$x_f^{(n)}$ converges uniformly to a continuous function $x_f$ on $[0, \tau_f(1/K - \e) + \tau_g(1/K - \e)] = [0, \tau_f(2/K - 2\e)].$ Furthermore,
\[
x_f^{(n)}(t + \tau) \lra x_f(\tau) + x_g(t),
\]
for $t \in [0, \tau_g(1/K - \e)].$

We can repeat this argument to see $x_f^{(n)}$ converges uniformly on $[0, \tau_f(N/K - N\e)]$
for any $N \in \N.$ Since $f$ is continuous it is bounded on compact time sets. Take $N$ 
large enough so that for any $T > 0$. Then $x_f^{(n)}$ converges uniformly on $[0, \tau_f(N/K - N\e)] \supset [0, T].$
\end{proof}


\section{Diffusions with local time dependent noise}\label{sec:existence_of_sde}
In this section we show that the map $f \mapsto (x_f, L)$ constructed in Section 2, when applied path-wise to a Brownian motion to yield the process $(x_B, L)$, is a construction of the process $(Y, L)$ in Theorem \ref{intro:th:sde_sigma_p}.
\begin{prop}\label{theorem:strong_existence_sec3}
Let $\sigma : \R_{\geq 0} \to \R_+$ be a Lipschitz function bounded away from zero, $(\Omega, \prob, (\mc{F}_t)_{t \geq 0})$ be a probability space supporting an $\mc{F}_t$-adapted Brownian motion $B,$ and take $X(0) \geq 0.$ There exists a continuous $\mc{F}_t$-adapted Markov process $(X, L)$ such that the following holds almost surely, for all $t \in [0, T]$:
\begin{align}
&X(t) = X(0) + \int_0^t\sigma(L(s)) \, \mr{d}B(s),\\
&\text{$L(0) = 0$, $L$ is nondecreasing and flat off $\{s: X(s) + L(s) = 0\}$}.
\end{align}Furthermore, 
\begin{align}\label{eq:quad_variation}
\langle X \rangle(t) = \int_0^t\sigma^2(L(s)) \, \mr{d}s.
\end{align}
\end{prop}
The proof of Theorem \ref{theorem:strong_existence_sec3} essentially follows from the results in Section \ref{sec:integ_constr}.
\begin{proof}
Almost every Brownian path is continuous, hence we can construct the pair
$(x^{(n)}_B, L^{(n)})$ by replacing $f$ in equation \eqref{def_x_f_n} pathwise
with the Brownian motion $B$. From Lemma \ref{lemma:lemma1} (vii), $x_B^{(n)}$
is the stochastic integral with respect to Brownian motion:
\[
x_B^{(n)} = x_B^{(n)}(0) + \int_0^t\sigma\Big(\frac{\lfloor nL^{(n)}(s)\rfloor}{n}\Big) \, \mr{d}B(s)
\]
where $L^{(n)}(t) := M^{x_B^{(n)}}(t).$ By Theorem \ref{theorem:convergence},
$(x_B^{(n)}, L^{(n)})$ converges in $C([0, T], \R^2)$ to the pair of processes $(x_B, L)$, almost surely.
This, along with Lemma \ref{lemma:lemma1} (x), implies $(x_B, L)$ is an $\mc{F}_t-$adapted Markov process.
Note that $L(t) = M^{x_B}(t)$ is an $\mc{F}_t$-adapted continuous nondecreasing process on $[0, T].$
Hence 
\[
\prob\left(\int_0^T\sigma^2(L(s)) \, \mr{d}s \leq T\sigma^2(L(T)) < \infty\right) = 1,
\]
and consequently we can define the stochastic integral
$\ds\int_0^t\sigma(L(s))\, \mr{d}B(s)$ \cite[Ch. 3.2]{KaratzasShreve}. Clearly
\begin{align*}
&\int_0^T\Big|\sigma\Big(\frac{\lfloor n L^{(n)}(s)\rfloor}{n}\Big) - \sigma(L(s))\Big|^2 \, \mr{d}\langle B\rangle(s)= \int_0^T\Big|\sigma\Big(\frac{\lfloor n L^{(n)}(s)\rfloor}{n}\Big) - \sigma(L(s))\Big|^2 \, \mr{d}s\\
&\leq T\Big\|\sigma\Big(\frac{\lfloor n L^{(n)}\rfloor}{n}\Big) - \sigma\circ L \Big\|_{[0, T]}^2\leq TK^2\Big\|\frac{\lfloor n L^{(n)}\rfloor}{n} - L\Big\|_{[0, T]}^2\\
&\leq  TK^2\Big(\|L^{(n)} - L\|_{[0, T]} + \frac{1}{n} \Big)^2\\
&\lra 0.
\end{align*}
Then, by classical results on convergence of stochastic integrals (\cite[Prop. 3.2.26]{KaratzasShreve}, for instance),
\begin{align*}
&\Big\|\int_0^t \sigma(L(s)) \, \mr{d}B(s) - \int_0^t\sigma\Big(\frac{\lfloor nL^{(n)}(s)\rfloor}{n}\Big) \, \mr{d}B(s) \Big\|_{[0, T]}\\
&= \Big\|X_B(0) + \int_0^t \sigma(L(s)) \, \mr{d}B(s) -  x_B^{(n)}\Big\|_{[0, T]}\\
&\lra 0, \text{ in probability.}
\end{align*}
But we already know 
\[
x_B^{(n)} \lra x_B, \text{ almost surely.}
\]
Therefore 
\[
x_B(t) = x_B(0) + \int_0^t\sigma(L(s)) \, \mr{d}B(s)
\]
for all $t \in [0, T]$, almost surely, where $L = M^{x_B}.$ Because $x_B$ can be constructed for any probability
space supporting a Brownian motion, and is adapted to the same filtration, this proves a strong solution exists.
\end{proof}

\begin{theorem} \label{th:reflected_diffusion}
\begin{align}\label{eq:reflected_SDE}
Y(t) := X(t) + L(t) = X(0)  + L(t) + \int_0^t\sigma(L(s)) \, \mr{d}B(s)
\end{align}
is a process reflecting inside $[0, \infty)$ with local time dependent noise. That is, $L$ is the local time of $Y$ at zero. Furthermore, there exists a continuous
random field $\Lambda(t, a): \R_{\geq 0} \times \R \to \R$ such that
\begin{enumerate}[label = (\roman*)]
\item $\Lambda(t, a)$ is $\mc{F}_t-$adapted for fixed $a.$
\item The map $t \mapsto \Lambda(t, a)$ is continuous and nondecreasing for each $a,$ with $\Lambda(0, a) = 0$
and $\Lambda(\cdot, a)$ flat off $\{s : Y(s) = a\}$.
\item For every Borel measurable $k: \R \to \R_{\geq 0}$,
\[
\int_0^tk(Y(s))\sigma^2(L(s)) \, \mr{d}s = 2\int_\R k(a)\Lambda(t, a)\, \mr{d}a
\]
almost surely, for all $t \in [0, T]$.
\end{enumerate}
We also have the It\^o-Tanaka formula for (linear combinations of) convex functions $f$:
\begin{align}\label{eq:gen_ito_rule}
f(Y(t)) = f(Y(0)) + \int_0^tf_-(Y(s))\sigma(L(s))\, \mr{d}B(s) + \int_\R\Lambda(t, a)\, \mr{d} \mu(a),
\end{align}
for all $t \in [0, T],$ almost surely.
Furthermore, $L(t)$ is the local time of $Y$ at zero $\Lambda(t, 0)$, almost surely, for all $t \in [0, T]$. Here $f_-$ is the left-hand derivative, which exists Lebesgue almost everywhere, and $\mu$ is the measure constructed from $\mu([a, b)) = f_-(b) - f_-(a)$.
\end{theorem}
\begin{remark}
See \cite[Chap 3.6D]{KaratzasShreve}, 
where Karatzas \& Shreve refer to \eqref{eq:gen_ito_rule} as the generalized It\^{o} rule for 
convex functions. The
random field $\Lambda$ is the stochastic field of local times where $\Lambda(t, a)$ is the local time of $Y$
at level set $a$ and time $t$. We use the same normalization of local time as Karatzas \& Shreve \cite[Remark 3.6.4 and Ch. 3.7]{KaratzasShreve}.
\end{remark}
\begin{proof}
The statements (i), (ii), and (iii) follow from \cite[Ch 3.7]{KaratzasShreve} since $X$ is a continuous martingale. That $L$ is the local time of $Y$ at zero, apply \eqref{eq:gen_ito_rule} to $f(x) = x^+ := \max\{x,0\}$, so that $f_-(x) = \mathds{1}_{(0, \infty)}$ and $\mu = \d_{\{0\}}.$ We have
\begin{align}\label{eq1:th:reflected_diffusion}
(Y(t))^+ &= (Y(0))^+ + \int_0^t\mathds{1}(Y(s) > 0)\sigma(L(s)) \, \mr{d}B(s) + \Lambda(t, 0)
\end{align}
for all $t \in [0, T],$ almost surely,
because the measure $\mu$ becomes a point mass at zero. Taking $k(x) = \mathds{1}_{\{0\}}$ in Theorem \ref{th:reflected_diffusion} (iii), we see 
\[
\int_0^t\mathds{1}(Y(s) = 0)\, \mr{d}\langle X\rangle(s) = 0,
\]
for all $t \in [0, T],$ almost surely. This implies that $\mathds{1}(Y(s) > 0)$ is equivalent to $\mathds{1}(Y(s) \geq 0)$ under the $L^2$ norm on $\Omega \times [0, T]$ generated by the measure $m(A) = \ex \int_0^t \mathds{1}_A(t, \omega)\,\mr{d}\langle X\rangle_s(\omega).$ Consequently the stochastic integrals 
\[
\int_0^t\mathds{1}(Y(s) > 0)\sigma(L(s)) \, \mr{d}B(s),
\] and 
\[
\int_0^t\mathds{1}(Y(s) \geq 0)\sigma(L(s)) \, \mr{d}B(s) = \int_0^t\sigma(L(s)) \, \mr{d}B(s) = X(t) - X(0)
\] are indistinguishable. That is, they agree almost surely, for all $t \in [0, T]$ \cite[Ch. 3.2]{KaratzasShreve}. Since $(Y(0))^+ = X(0) > 0$ we write \eqref{eq1:th:reflected_diffusion} as
\begin{align}\label{eq2:th:reflected_diffusion}
(Y(t))^+ &= (Y(0))^+ + X(t) - X(0) + \Lambda(t, 0)\\
&= X(t) + \Lambda(t, 0),
\end{align}
almost surely, for all $t \in [0, T]$.
Now $(Y(t))^+ := \max\{Y(t), 0\}= Y(t)$ since $Y \geq 0$, so
\[
Y(t) = X(t) + \Lambda(t, 0),
\]
almost surely, for all $t \in [0, T]$.
But by definition,
\[
Y(t) := X(t) + L(t).
\]
Hence
$L(t)$ and $\Lambda(t, 0)$ are indistinguishable. Consequently $L$ is the martingale local time of $Y$ at zero, justifying our claim that $Y$ is a reflected diffusion with local time dependent noise.
\end{proof}

\begin{remark}\label{remark:existence_tau_f}
By examining the proof of Theorem \ref{theorem:convergence}, $x_f^{(n)}$ converges to $x_f$ uniformly on $[0, \tau_f(C)]$ for any finite $C > 0.$ This holds even when $\tau_f(C) = \infty$ since $x_f(t) = x_f(t^*) + \sigma(L(t^*))(f(t) - f(t^*))$ for all $t > t^* = \inf_{s \in \R}\{M^f(s) = \max_{x \in \R} M^f(x)\}$. This does not apply for Brownian paths, however, since $\tau_B(C) < \infty$ almost surely.
\end{remark}

\subsection{Time of determinacy}
In this subsection we give the proof of Theorem \ref{intro:th:sde_sigma_p} after introducing  various lemmas. We consider the pair $(Y, L)$ when the noise of $Y$  
is a function $\sigma \geq 0$ with support $[0, \sigma_0]$ and locally Lipschitz on $[0, \sigma_0).$ Such a function is not bounded away from zero, so existence of
$(Y, L)$ does not immediately follow from Theorems \ref{theorem:strong_existence_sec3} and \ref{th:reflected_diffusion}. However,
\[
 \sigma_{\d}(x) =
   \begin{dcases}
     \sigma(x) &\text{ for $x \in [0, \sigma_0-\d]$}\\
     \sigma(\sigma_0 - \d) &\text{ for $x > \sigma_0 - \d$},
   \end{dcases}
\]
does satisfy these properties for each $\d > 0,$ and we can build up a strong solution by letting $\d$ decrease to zero. It is conceivable that the noise will disappear completely if the noise decreases to zero at the random time the local time reaches $\sigma_0,$ the edge of the support of $\sigma.$ After this
time we may continuously extend the process to be zero. That is, it will be deterministically zero from this time forward. Because of this we call this time the \emph{time of determinacy} and denote it by $\tau_{\sigma_0}$. Formally,
\[
\tau_{\sigma_0} = \inf\{t > 0 : \sigma(L(t)) = 0\} = \inf\{t > 0: L(t) \geq \sigma_0\}.
\]
It is evident that $\tau_{\sigma_0}$ is a hitting time of $L$. We characterize the distribution of hitting times of $L$ in the statement of Theorem \ref{intro:th:sde_sigma_p}.
\noindent
In order to stay true to our notation of $\tau_f$ defined earlier, the next proposition is phrased using this notation.

\begin{lemma}\label{lemma:cont_hitting_time}
Let $a_n$ be a deterministic sequence of non-negative real numbers converging to $a \in \R_{\geq 0}$. Let
\[
\tau_B(a_n) = \inf\{t > 0: B(t) < -a_n\}
\] be the first time the running minimum of a standard Brownian motion exceeds $a_n.$ Then $\tau_B(a_n) \to \tau_B(a)$, almost surely.
\end{lemma}
\begin{proof}
Let 
\begin{align*}
\tau^+ &:= \limsup_{n \to \infty}\tau_B(a_n),\\
\tau^- &:= \liminf_{n \to \infty}\tau_B(a_n).
\end{align*}
Note $\tau^+$ and $\tau^-$ are stopping times, and $\tau^+ \leq \tau_B(\max_n a_n) < \infty$, almost surely. For each $\omega \in \Omega$ there is a random sequence $n_k(\omega)$ such that
$\tau_{B_\omega}(a_{n_k(\omega)}) \lra \tau^-(\omega).$ By continuity
\[
B(\tau^-) = \lim_{k \to \infty} B(\tau_B(a_{n_k(\omega)})) = \lim_{k \to \infty}-a_{n_k(\omega)} = -a.
\]
By the strong Markov property $\{B(t + \tau^-) + a : t \geq 0\}$ is a standard Brownian motion, so 
for every $\e > 0$ there is an $s \in (\tau^-, \tau^- + \e)$ such that $B(s + \tau^-) + a < 0$ (i.e. $B(s + \tau^-) < -a$), almost surely. This implies $\tau_B(a) \leq \tau^-$, almost surely. 

Pick $\e > 0$ and choose $N$ such
that $a_n < a + \e$ for all $n > N.$ Then $\tau_B(a_n) < \tau_B(a + \e)$, almost surely, for all $n > N.$ Consequently
\[
\tau^+ = \limsup_{n \to \infty}\tau_B(a_n) \leq \tau_B(a + \e),
\]
almost surely. Recall that $z \mapsto \tau_B(z)$ is right continuous, almost surely. For a fixed $\eta > 0$ one can find $\e_\eta$ such that $\tau^+ \leq \tau_B(a + \e_\eta) < \tau_B(a) + \eta$ with probability greater than $1 - \eta.$ By letting $\eta \to 0$, we see $\tau^+ \leq \tau_B(a)$, almost surely. Hence
\[
\tau^+ \leq \tau_B(a) \leq \tau^-
\]
almost surely, so $\tau_B(a_n) \lra \tau_B(a)$ almost surely, as well.
\end{proof}
\begin{lemma}\label{lemma:time_change}
Let $\sigma$ satisfy the conditions in Theorem \ref{theorem:strong_existence_sec3}, choose $\alpha \in [0, \sigma_0)$ and $C >0$ such that $\int_0^\alpha \frac{\md s}{\sigma(s)} < C.$ By Remark \ref{remark:existence_tau_f} we can define the process $(Y, L)$ on $[0, \tau_B(C)]$. Then
\[
L\left(\tau_{X(0) + B}\left(\int_0^\alpha \frac{\mr{d}s}{\sigma(s)}\right)\right) = \alpha.
\]
\end{lemma}
\begin{remark}
This can be thought of as a time change formula for the local time.
\end{remark}
\begin{proof}
According to Lemma \ref{lemma:lemma1} (iii), $L^{(n)}(t_{\lfloor \alpha n\rfloor}^{(n)}) = \lfloor \alpha n\rfloor/n,$ 
where 
\[
t_{\lfloor \alpha n\rfloor}^{(n)} = \tau_{X(0) + B}\left(\sum_{j = 0}^{\lfloor \alpha n\rfloor}\frac{1}{\sigma(j/n)n}\right).
\]
Note that 
\[
\sum_{j = 0}^{\lfloor \alpha n\rfloor}\frac{1}{\sigma(j/n)n} \lra  \int_0^\alpha \frac{\mr{d}s}{\sigma(s)}
\]
since the sum is a Riemann approximation converging to the integral.
By Lemma \ref{lemma:cont_hitting_time},
\[
t_{\lfloor \alpha n\rfloor}^{(n)} \lra \tau_{X(0) + B}\left(\int_0^\alpha\frac{\mr{d}s}{\sigma(s)} \right) =: t(\alpha),
\]
almost surely. Since $L^{(n)}$ converges to $L$ uniformly on $[0, \tau_{X(0) + B}(C)] \supset [0, t(\alpha)]$, almost surely, and
\[
L(t(\alpha)) = \lim_{n\to \infty}L^{(n)}(t_{\lfloor \alpha n\rfloor}^{(n)}) = \lim_{n \to \infty}\lfloor \alpha n\rfloor/n = \alpha,
\] almost surely.
\end{proof}

\begin{proof}[Proof of Theorem \ref{intro:th:sde_sigma_p}]
\emph{Step 1:} We show strong existence holds. 
We assume $\sigma_0 < \infty$, in the case $\sigma_0 = \infty$ one replaces $\sigma_0 - \d$ below with $N_\d$, where $N_\d \to \infty$ as $\d \to 0.$
Set
\[
S(\alpha) := \int_0^\alpha\frac{\mr{d}s}{\sigma(s)}
\]
for $\alpha \in [0, \sigma_0).$
Pick $\d > 0$, we can 
construct a solution $\{(Y(t), L(t)) : t \in [0, \tau_B(S(\sigma_0 - \d))]\}$ where
\[
Y(t) = Y(0) + L(t) + \int_0^t\sigma_\d(L(s))\, \md B(t).
\]
By Lemma \ref{lemma:time_change},
\[
L(\tau_{Y(0) + B}(S(\sigma_0 - \d))) = \sigma_0 - \d,
\]
so $L(s)$ ranges from zero to $\sigma_0 - \d$ on $[0, \tau_{Y(0) + B}(S(\sigma_0 - \d))].$
This implies that $\sigma_{\d}\circ L = \sigma$ on $[0, \tau_{Y(0) + B}(S(\sigma_0 - \d))]$. So
\begin{align}\label{eq1:proof_SDE_decrease}
Y(t) = Y(0) + L(t) + \int_0^t\sigma(L(s))\, \md B(t)
\end{align}
for $t \in [0, \tau_{Y(0) + B}(S(\sigma_0 - \d))].$ This pathwise construction can be done for any $\d > 0$.
Take the sequence $\d_n = 1/n$ and for a given sample path of Brownian motion define
\begin{align}\label{eq2:proof_SDE_decrease}
Y_n(t) = x(0) + L(t)+ x_B(t)
\end{align}
where $(x_B, L)$ is constructed according to Theorem \ref{theorem:convergence} for \(t \in [0, \tau_{Y(0) + B}\) \((S(\sigma_0 - 1/n))]\). For $n < m$ the paths $Y_n$ and $Y_m$ are identical,
almost surely, on their common domain $[0, \tau_{Y(0) + B}(\sigma_0 - 1/n)]$. By taking $n \to \infty$ a 
strong solution $(Y, L)$ of \eqref{eq2:proof_SDE_decrease} can be defined on \([0, \tau_{Y(0) + B}(S(\sigma_0-)))\), where 
\begin{align}\label{eq3:proof_SDE_decrease}
S(\sigma_0-) := \lim_{\d \to 0}S(\sigma_0 - \d) = \int_0^{\sigma_0}\frac{\md s}{\sigma(s)},
\end{align}
by setting $Y(t) = \lim_{n \to \infty}Y_n(t)$ for $t \in [0, \tau_{Y(0) + B}(S(\sigma_0-)))$. This shows a strong
solution to \eqref{eq1:proof_SDE_decrease} holds for $t \in [0, \infty)$ when $\int_0^{\sigma_0}\frac{\md s}{\sigma(s)} = \infty.$ Furthermore, Lemma \ref{lemma:time_change} yields
\[
L(\tau_{Y(0) + B}(S(\sigma_0 - 1/n))) = \sigma_0 - 1/n,
\]
so the noise of $Y$ at time  $\tau_{Y(0) + B}(S(\sigma_0 - 1/n))$ is
\[
\sigma(L(\tau_{Y(0) + B}(S(\sigma_0 - 1/n)))) = \sigma(\sigma_0 - 1/n).
\]
Consequently, the noise of $Y(t)$ approaches zero as $t$ approaches $\tau_{Y(0) + B}(\int_0^{\sigma_0}\frac{\md s}{\sigma(s)})$. Setting $Y(t) = 0$ for $t \in [\tau, \infty)$, $Y$ solves \eqref{eq1:proof_SDE_decrease} for all $t \in [0, \infty)$ as long as we 
can sensibly define $L$ for these times. One option would be to freeze $L(t)$ at $\sigma_0$ after time $\tau$, but then $L$ would not retain the notion of local time because the 
local time of a constant function (at that level) is infinite. It seems more natural to define 
$L(t) = \infty$ for $t \in (\tau, \infty)$, meaning $L$ would jump to infinity. With these definitions, $(Y, L)$ solves \eqref{eq1:proof_SDE_decrease} by extending $\sigma(\infty) = 0$. 

\emph{Step 2:} We characterize the Laplace transform for the hitting time of $L$. For any $a \in [0, \sigma_0]$, the argument in the previous step shows
\[
L(\tau_{Y(0) + B}(S(a-))) = a.
\]
That is,
\[
\tau_a := \inf\{t > 0 : L(t) \geq a\} = \tau_{Y(0) + B}S(a-),
\]
almost surely. This characterizes the hitting times of $L.$ The Laplace transform of $Y(0) + B$
is well known \cite{KaratzasShreve}. Since $Y(0)$ is independent of the Brownian path, using the strong Markov property we have
$\tau_{Y(0) + B}(a) \dist \tau_W(Y(0)) + \tau_B(a)$ where $W$ and $B$ are two independent standard Brownian motions independent of $Y(0)$. The Laplace transform is then computed
\begin{align*}
\phi(\ld) &= \ex(\exp{-\ld \tau_a})\\
&= \ex \exp\left\{-\ld(\tau_W(Y(0)) + \tau_B(S(a-))\right\}\\
&= \ds \ex \exp\left\{-\sqrt{2\ld}\, Y(0)\right\}\ex \exp\left\{-\sqrt{2\ld}\ds\int_0^a\frac{\md s}{\sigma(s)}\right\}
\end{align*}


\end{proof}

\section{Propagation of chaos for systems of diffusions}\label{sec:systems}
In this section we characterize the macroscopic behavior for systems of 
reflected diffusions interacting through the local time of the processes in the system in the prescribed manner.
For a given $n$, let
\((\Omega, \prob,\) \((\mc{F}_t)_{t \geq 0})\) be a 
probability space supporting $n$ independent $\mc{F}_t-$adapted 
Brownian motions $B_i,$ for  $i =1, \dots, n.$ For any Lipschitz
function $\sigma: [0, \infty) \to [0, \infty),$ we consider continuous
$\mc{F}_t$-adapted processes $(X_i^{(n)}, L_i^{(n)})$ such that \eqref{eq:system_intro} holds for all $t \in [0, T]$, almost surely.

We assume that $X_i^{(n)}(0)$ are i.i.d.\ and independent of the Brownian motions. In the convergence theorems we also assume there exist i.i.d.\ $X_i(0)$ such that \(W_1(X_i^{(n)}(0)\) \(, X_i(0))\to 0\). Here $W_1$ is the Wasserstein-1 metric. Recall that convergence in this metric is equivalent to the existence of a probability space supporting $X_i^{(n)}(0), X_i(0)$ such that $X_i^{(n)}(0)$ converges almost surely and in $L_1$ to $X_i(0).$

\begin{theorem}(Strong existence and uniqueness)\label{th:system_existence}
Let $\sigma:[0, \infty) \to (0, \infty)$ be Lipschitz, and $(\Omega, \prob, (\mc{F}_t)_{t \geq 0})$ a probability space supporting $n$ i.i.d.\
$\mc{F}_t-$Brownian motions. There exist continuous $\mc{F}_t-$adapted processes $(X_1^{(n)}, \dots, X_n^{(n)}, L_1^{(n)}, \dots, L_n^{(n)})$ satisfying \eqref{eq:system_intro} for all $t \in [0, T]$, almost surely.
\end{theorem}
\begin{remark}\label{remark:system_existence}
The existence and uniqueness can be shown from Dupuis and Ishii's work as mentioned in Remark \ref{intro:remark}. Similar to the $n = 1$ case, Skorohod's lemma shows that $L_i^{(n)}$ is the running minimum of
$X_i^{(n)}(0) + \int_0^t\sigma(\mb{L}^{(n)}(s))\,\mr{d}B_i(s)$ below zero. Also note that the system is 
clearly exchangeable.
\end{remark}
\begin{lemma}\label{lemma:L_2bound}
For every  $T < \infty,$ \[
\sup_n\ex(\mb{L}^{(n)}(T))^2 < 4\sigma(0)(\sigma(0) + K)Te^{4(K^2 + \sigma(0))T}.\]
\end{lemma}
\begin{proof}
By definition, 
\begin{align}
\begin{split}
[\mb{L}^{(n)}(T)]^2 = \Big(\frac{1}{n}\sum_{i = 1}^nL_i^{(n)}\Big)^2
&= \frac{1}{n^2}\left(\sum_{i = 1}^nL_i^2(T) + \sum_{i \neq j}L_i(T)L_j(T) \right).
\end{split}
\end{align}
Exchangeability implies $\ex (L_i(T))^2 = \ex (L_1(T))^2$ and $\ex L_i(T)L_j(T) = \ex L_1(T)L_2(T)$ for $i \neq j.$ Taking expectations,
\begin{align}
\begin{split}
\ex(\mb{L}^{(n)}(T))^2 &\leq \frac{1}{n^2}\left(n\ex (L_1^{(n)}(T))^2 + n(n-1)\ex (L_1(T)L_2(T)) \right), \text{ by Cauchy-Schwarz},\\
&\leq \frac{1}{n^2}\Big(n\ex (L_1^{(n)}(T))^2 + n(n-1)\ex (L_1^{(n)}(T))^2 \Big)\\
&= \ex (L_1^{(n)}(T))^2\\
&\leq \ex \sup_{s \in [0, T]}\left|\int_0^s\sigma(\mb{L}^{(n)}(z))\, \mr{d}B_i\right|^2.
\end{split}
\end{align}
Doob's maximal inequality yields
\begin{align}\label{lemma_sup_L_2:eq1}
\begin{split}
&\ex \sup_{s \in [0, T]}\left|\int_0^s\sigma(\mb{L}^{(n)}(z))\, \mr{d}B_i\right|^2\leq 4\ex\left(\int_0^T\sigma(\mb{L}^{(n)}(z))\, \mr{d}B_i(z)\right)^2\\
&= 4\ex \int_0^T\sigma^2(\mb{L}^{(n)}(s))\, \mr{d}s\leq 4\ex \int_0^T(\sigma(0) + K\mb{L}^{(n)}(s))^2\, \mr{d}s\\
&\leq 4\ex \int_0^T \sigma(0)(\sigma(0) + K) + (K^2 + \sigma(0))(\L^{(n)}(s))^2\, \md s\\ 
&\leq 4\sigma(0)(\sigma(0) + K)T + 4(K^2 + \sigma(0))\int_0^T\ex(\mb{L}^{(n)}(s))^2\, \mr{d}s.
\end{split}
\end{align}
If we denote $\alpha_n(t) = \ex (\mb{L}^{(n)}(t))^2$, \eqref{lemma_sup_L_2:eq1} implies
\[
\alpha_n(t) \leq 4\sigma(0)(\sigma(0) + K)t + 4(K^2 + \sigma(0))\int_0^t\alpha_n(s)\, \mr{d}s.
\]
Then by Gr\"onwall's inequality,
\[
\ex (\mb{L}^{(n)}(t))^2 = \alpha_n(t) \leq 4\sigma(0)(\sigma(0) + K)te^{4(K^2 + \sigma(0))t}.
\]
\end{proof}
\noindent
Recall that
\[
\omega(f, \delta, T) := \sup_{\substack{|t - s| < \d \\ s, t \in [0, T]}}|f(t) - f(s)|
\]
is the modulus of continuity of $f$ on $[0, T]$ depending on $\d.$

\begin{lemma}[See \cite{FischerNappo}]\label{lemma:Mean_mod_cont}
For every $p > 0$ there exists a constant $C_p$ independent of $\d, T$, such that
\[
\ex (\omega(B, \d, T))^p < C_p \left(\d\log\frac{T}{\d} \right)^{p/2}.
\]
\end{lemma}
\begin{lemma}\label{lemma:tightness}
The collection of processes $\{\L^{(n)} : n \in \N\}$ is tight in $C[0, T].$
\end{lemma}
\begin{proof}
Let $\theta_n(t) = \int_0^t\sigma^2\big(\L^{(n)}(s)\big)\, \mr{d}s$, so by a time change there are standard Brownian motions $\wt{B}_i$ such that
\begin{align}\label{eq:time_change}
\wt{B}_i\circ \theta_n \dist \int_0^\cdot \sigma\big( \L^{(n)}(s)\big)\, \mr{d}B_i(s) =: Y_i^{(n)}(\cdot).
\end{align}
Notice that
\[
\omega(\L^{(n)}, \d, T) = \omega\Big(\frac{1}{n}\sum_{i = 1}^nL_i^{(n)}, \d, T \Big)
\leq \frac{1}{n}\sum_{i = 1}^n\omega(L_i^{(n)}, \d, T).
\]
According to the Skorohod Lemma, 
\begin{align*}
L_i^{(n)}(t) &:= \min_{0 \leq s \leq t}-\Big(X_i^{(n)}(0) + \int_0^s\sigma\big(\L^{(n)}(z)\big)\, \mr{d}B_i(z) \Big) \lor 0\\
&= \min_{0 \leq s \leq t}-\Big(X_i^{(n)}(0) + Y_i^{(n)}(s)\Big) \lor 0,
\end{align*}
so that
\[
\omega(L_i^{(n)}, T, \d) \leq \omega(Y_i^{(n)}, T, \d),
\]
almost surely.
Also
\begin{align*}
|\theta_n(t) - \theta_n(s)| &= \left|\int_s^t \sigma^2(\L^{(n)}(z))\,\mr{d}z\right|\\
&< 4[\sigma^2(0) + K\sigma(0) + (K^2 + 1 + \sigma(0))(\L^{(n)}(T))^2]|t - s|
\end{align*}
almost surely. Consequently,
\[
\omega(\wt{B}_i\circ \theta_n, \d, T) < \omega(\wt{B}_i, (\alpha + \beta (\L^{(n)}(T))^2)\d, (\alpha + \beta (\L^{(n)}(T))^2)T)
\]
almost everywhere, where $\alpha = 4(\sigma^2(0) + K\sigma(0)), \beta = 4(K^2 + 1 + \sigma(0))$.
For any fixed $M, \e, \d > 0$, Chebyshev's inequality gives
\begin{align}\label{lemma_tightness:eq1}
\begin{split}
&\prob(\omega(\L^{(n)}, \d, T) > \e) \\
&\leq \prob\left(\frac{1}{n}\sum_{i = 1}^n\omega(B_i\circ \theta_n, \d, T) > \e, \ \L^{(n)}(T) < M \right) + \prob(\L^{(n)}(T) > M)\\
&\leq \prob\left(\frac{1}{n}\sum_{i = 1}^n\omega(B_i, (\alpha + \beta M^2)\d, (\alpha + \beta M^2)T) > \e\right) + M^{-1}\ex\, (\L^{(n)}(T))\\
&\leq \e^{-1}\ex \, \omega(B_1, (\alpha + \beta M^2)\d, (\alpha + \beta M^2)T) + M^{-1}\ex\, (\L^{(n)}(T)), \text{ apply Lemma \ref{lemma:Mean_mod_cont}},\\
&\leq \e^{-1}C\left((\alpha + \beta M^2)\d \log\frac{T}{\d}\right)^{1/2} + M^{-1}\ex\, (\L^{(n)}(T))\\
&\leq \e^{-1}C\left((\alpha + \beta M^2)\d \log\frac{T}{\d}\right)^{1/2} + M^{-1}\sup_n\sqrt{\ex \, (\L^{(n)}(T))^2}.
\end{split}
\end{align}
Taking $\lim_{\d \to 0}\sup_{n}$ of both sides, and using Lemma \ref{lemma:L_2bound}, we see
\[
\lim_{\d \to 0}\sup_n \prob(\omega(\L^{(n)}, \d, T) > \e) = 0
\]
for any $\e, T > 0.$ Since $\L^{(n)}(0) = 0$ almost surely, this is sufficient for tightness.
\end{proof}
\noindent
Recall Definition \ref{defn:tau_f}.
We denote $\tau_{\frac{1}{n}\sum_{i = 1}^nM^{X_i^{(n)} + B_i}}(a)$ as $\tau_{\ol{M}_n}(a)$ for convenience. The Lemma below is an analog of Lemma \ref{lemma:time_change}, and the proof is similar.
\begin{lemma}\label{lemma:hitting_time} For each $a > 0$,
\[
\L^{(n)}\left(\tau_{\ol{M}_n}\left(\int_0^a\frac{\mr{d}s}{\sigma(s)}\right)\right) = a,
\]
almost surely.
\end{lemma}

We use the following classical lemma whose proof remains for the reader.
\begin{lemma}\label{lemma:hitting_time_for_limit}
Let $f_n$ be a sequence in $C[0, T]$ converging uniformly to a function $f$. If $f$ is strictly increasing, then for any $a \in [f(0), f(T))$, 
\[
\tau_{f_n}(a) \lra \tau_f(a).
\]
\end{lemma}
\qed

\begin{prop}
\label{prop:deter_limit}
$\L^{(n)}(\cdot)$ converges in distribution to a deterministic continuous function $\ell(\cdot)$
such that
\[
\ell \left(\tau_{\ex M^{B + X_i(0)}(t)}\left(\int_0^a\frac{\md s}{\sigma(s)} \right)\right) = a.
\]
In particular, when $X_i(0) = 0$, $\ex M^B(t) = \sqrt{2t/\pi}$ and therefore
\[
\ell \left(\frac{\pi}{2}\left(\int_0^a\frac{\mr{d}}{\sigma(s)}\right)^2\right) = a.
\]
\end{prop}
\begin{proof}
By tightness of $\L^{(n)}$, let $\L^{(n_k)}$ be some sequence of processes converging in distribution to a process $L$. By the SLLN,
\[
\frac{1}{n}\sum_{i = 1}^nM^{B_i + X_i^{(n)}(0)}(t) \lra \ex M^{B_1 + X_1(0)}(t)
\]
almost surely on $C[0, T]$. Since $\ex M^{B_1 + X_1(0)}(t)$ is strictly increasing, Lemma \ref{lemma:hitting_time_for_limit} implies
\[
\tau_{\ol{M}_n}(b) \to \tau_{\ex M^{B_1 + X_1(0)}}(b)
\]  almost surely, for each $b \in \left[0, \int_0^A\frac{\mr{d}s}{\sigma(s)}\right)$ where $A$ is chosen so $\int_0^A\frac{\mr{d}s}{\sigma(s)} = \ex M^{B + X_i(0)}(T).$ From Lemma \ref{lemma:hitting_time} and the convergence of $\L^{(n_k)}$
\begin{align}\label{eq1:deter_limit}
L\left(\tau_{\ex M^{B + X_i(0)}(\cdot)}\left(\int_0^a\frac{\md s}{\sigma(s)} \right)\right)
\dist \lim_{n_k \to \infty}\L^{(n_k)}\left(\tau_{\ol{M}_n}\left(\int_0^a\frac{\md s}{\sigma(s)} \right)\right) = a.
\end{align}

for each $a \in [0, A).$ Consequently, 
 \[
L\left(\tau_{\ex M^{B + X_i(0)}(\cdot)}\left(\int_0^a\frac{\md s}{\sigma(s)} \right)\right) = a,
\]
for all rational values in $[0, A)$, almost surely. Almost every path of $L$ is continuous and nondecreasing. Hence, almost every path has a unique nondecreasing and right continuous right inverse determined from its values on a dense set. By \eqref{eq1:deter_limit}, this 
right inverse is deterministic and not dependent on the subsequence $n_k$. Therefore $L$ is deterministic and does not depend on the subsequence $n_k$. When $X_1(0) = 0$, we know $\ex M^B(t) = \sqrt{2t/\pi}$, so the right inverse is in fact equal to 
\[
a \mapsto \frac{\pi}{2}\left(\int_0^a\frac{\mr{d}}{\sigma(s)}\right)^2.
\]
\end{proof}

We now prove Theorem \ref{progagation_chaos}.
\begin{proof}[Proof of Theorem \ref{progagation_chaos}]
As in \eqref{eq:time_change}, $X_i^{(n)}$ can be written as a time change of a reflected Brownian motion $Z_i:$
\[
X_i^{(n)} \dist Z_i \circ \theta_n, \, \theta_n(t) = \int_0^t\sigma^2\big(\L^{(n)}(s) \big)\md s.
\] By Proposition \ref{prop:deter_limit}
this time change $\theta_n(\cdot)$ converges in distribution to a deterministic function $\int_0^\cdot \sigma^2(\ell(s))\md s.$ Consequently
any finite collection of particles $(X_1^{(n)}, \dots, X_k)$ will converge to 
$(\wt{X}_1, \dots, \wt{X}_k)$ given by
\begin{align}\label{eq:time_change_limit}
&\wt{X}_i(\cdot) = X_i(0) + Z_i \circ \int_0^\cdot \sigma^2(\ell(s)) \md s,\\
&L_i(t) = \text{ local time at zero of $\wt{X}_i$ at time $t.$}
\end{align}
where the $Z_i$ are independent reflected Brownian motions and $X_i(0)$ is the distributional limit of $X_i^{(n)}(0)$. Therefore $(X_1^{(n)}, \dots, X_k)$ converges to a collection of independent processes. This demonstrates propagation of chaos for the system.

To show the remaining claims, we use a standard localization argument by first assuming 
$\sigma$ to be bounded.
The convergence assumption on $X_i^{(n)}$ imply the existence of a probability space so $X_i^{(n)}(0) \lra X_i(0)$ almost surely and in $L_1$.
By Proposition \ref{prop:deter_limit}, $\L^{(n)}$ converges in probability to 
$\ell \in C[0, T].$ Consequently we have a subsequence $n_m$ such that
$\L^{(n_m)} \lra \ell$, almost surely. From boundedness of $\sigma$ it follows that
$\|\sigma\circ \L^{(n_m)} - \sigma\circ \ell\|_{[0, T]} \to 0$ almost surely and in $L_2.$ We have
\begin{align*}
&\|X_i^{(n)} - \wt{X}_i^{(n)}\|_{[0, T]}\\
 &= \sup_{t \in [0, T]}\left|L_i^{(n)}(t) + X_i^{(n)}(0) + \int_0^t\sigma(\L^{(n)}(s))\, \md B_i(s)
 - \wt{L}_i(t) - X_i(0) - \int_0^t\sigma(\ell(s))\,\md B_i(s)\right|\\
&\leq |X_i^{(n)}(0) - X_i(0)| + \|L_i^{(n)} - \wt{L}_i\|_{[0, T]} + \sup_{t \in [0, T]}\left|\int_0^t\sigma(\L^{(n)}(s)) - \sigma(\ell(s))\,\md B_i(s)\right|\\
&\leq |X_i^{(n)}(0) - X_i(0)| + 2\sup_{t \in [0, T]}\left|\int_0^t\sigma(\L^{(n)}(s)) - \sigma(\ell(s))\,\md B_i(s)\right|.
\end{align*}
Recall $\ex |X_i^{(n)}(0) - X_i(0)| \to 0.$ Taking expectations, applying Cauchy-Schwarz, and using Doob's maximal inequality,
\begin{align}
\begin{split}\label{eq:POC_exp_bounds}
\ex\|X_i^{(n_m)} - \wt{X}_i\|_{[0, T]} &\leq  
\ex|X_i^{(n_m)}(0) - X_i(0)| + 2\sqrt{\ex \int_0^T|\sigma(\L^{(n_m)}(s)) - \sigma(\ell(s))|^2\,\md s}\\
&\leq \ex|X_i^{(n_m)}(0) - X_i(0)| + 2\sqrt{T\ex\|\sigma\circ \L^{(n_m)} - \sigma\circ \ell\|_{[0, T]}^2}\\
&\lra 0.
\end{split}
\end{align}
(This also shows that $(X_1^{(n)}, \dots, X_k^{(n)})$ converges as a distribution on $C([0, T], \R^k)$
to the $k-$tuple of independent processes $(\wt{X}_1, \dots, \wt{X}_k)$.)
\end{proof}

Consequently, by the triangle inequality and exchangeability of the system,
\begin{align*}
\ex \|\L^{(n_m)} - \frac{1}{n_m}\sum_{i = 1}^{n_m}\wt{L}_i\|_{[0, T]} \leq \ex \|X_i^{(n_m)} - \wt{X}_i\|_{[0, T]} \lra 0.
\end{align*}
But by the SLLN
\[
\frac{1}{n}\sum_{i = 1}^n\wt{L}_i(t) \lra \ex \wt{L}_1(t),
\]
for each fixed $t,$ almost surely. So $\L^{(n_m)}(\cdot)$ converges to $\ex \wt{L}_1(\cdot)$ in $L_1$. Since $\L^{(n_m)}(\cdot)$ already converges in probability to $\ell(\cdot),$ we conclude
$l(t) = \ex \wt{L}_1(t)$ for all $t \in [0, T]$, almost surely.

\subsection{Hydrodynamic limit: Measure on path-space vs. measure-valued paths}
In this subsection we compare two different ways of viewning the empirical collection of particles $(X_1^{(n)},\dots, X_n^{(n)})$: That of a measure on path-space (scheme A), and that of measure-valued paths (scheme B).
We first describe scheme A.
For each $n,$ $\omega \mapsto X_i^{(n)}(\omega)$ maps the probability space
$\Omega$ to $C[0, T].$ Consequently, $\omega \mapsto \delta_{X_i^{(n)}(\omega)}$ is a map
from $\Omega$ to $\mc{P}(C[0, T])$, the space of probability measures on $C[0, T]$ with the metric of weak convergence. Here $\delta_g$ is the point mass at $g \in C[0, T].$ Therefore,
\[
\pi^{(n)}:= \frac{1}{n}\sum_{i = 1}^n\delta_{X_i^{(n)}}
\]
is a random element in $\mc{P}(C[0, T]).$ In other words, $\pi^{(n)}$ is a random measure on path space.

In scheme B we view $\pi^{(n)}$ as a measure-valued path. Note that for a fixed time $t \in [0, T]$ $\{X_i^{(n)}(t) : i = 1, \dots, n\}$ is a random collection of particles in $\R$. So
\[
\pi^{(n)}_t := \frac{1}{n}\sum_{i = 1}^n\delta_{X_i^{(n)}(t)}
\]
is a (random) element in $\mc{P}(\R),$ the space of probability measures on $\R$ with the metric of weak convergence. It follows easily from path-wise continuity of the $X_i^{(n)}$ that $\pi^{(n)}_t$ is a.s.\ continuous in $t$ under this metric. Therefore $\pi^{(n)}_t$ is random element in
$C([0, T], \mc{P}(\R))$, i.e. a continuous $\mc{P}(\R)$-valued process, and induces a measure on
continuous measure-valued paths.

\begin{theorem}(Hydrodynamic Limit A)\label{HL:A}
Let $Q$ denote the law on $C[0, T]$ induced by the process $\wt{X}_1 = \lim_{n \to \infty} X_1^{(n)}$. Then
$\pi^{(n)}$ converges weakly to $\d_Q$. That is,
\[
\pi^{(n)} \implies \d_Q
\]
in $\mc{P}(C[0, T]).$
\end{theorem}

\begin{theorem}(Hydrodynamic Limit B)\label{HL:B}
Let $q(t, x)$ be the density of $\wt{X}_1$. The empirical process $\pi^{(n)}$ converges 
in distribution to $\{q(t, x)\md x : t \in [0, T]\}$ in $C([0, T], \mc{P}(\R))$, where $q$ is the transition density of the process $\wt{X}_1.$
\end{theorem}

\noindent
Theorem \ref{HL:A} is equivalent to the propagation of chaos (Theorem \ref{progagation_chaos}) for exchangeable systems of processes. See M\'eleard \cite{Meleard}, where the difference
between settings A and B is discussed, and an (easily adapted) argument shows that Theorem \ref{HL:A}
implies Theorem \ref{HL:B}.

\section{Existence and uniqueness of the PDE}\label{subsec:E_U} The limit process of $\wt{X}_1$ admits a density because it is a continuously differentiable time change of 
a reflected Brownian motion as given in equation \eqref{eq:time_change_limit}. 
Let $q(t, x)$ be the density of $\wt{X}_1$ with initial condition $X_1(0).$ From the generalized Ito rule with convex functions \cite[Chap 3.6D]{KaratzasShreve}, we see from \eqref{eq:time_change_limit} that
\[
\int_0^t \frac{\mathds 1{(0 < Y(s) < \e)}}{\e} \sigma^2(\ell(s))^2 \md s = 2\e^{-1}\int_0^\e \Lambda(t, a)\md a,
\] where $\Lambda(t, a)$ is the random field local time of $\wt{X}_1$, $\Lambda(t, 0) = \wt{L}_1(t)$. By taking expectations of both sides, letting $\e \to 0$, one can show using a dominated convergence argument that
\begin{align}
\begin{split}\label{eq:expected_LT}
2\ex \wt{L}_1(t) &= \int_0^tq(s, 0)\, \md \langle \wt{X}_1 \rangle(s) = \int_0^tq(s, 0)\sigma^2(\ell(s))\, \md s,
\end{split}
\end{align}
so
\begin{align}\label{eq:integral_rep_ell}
\ell(t) = \ex \wt{L}_1(t) = \frac{1}{2}\int_0^tq(0, s)\sigma^2(\ell(s))\,\md s.
\end{align} 
Then the pair $(q, \ell)$ solves the PDE given in subsection \ref{intro:PDE},
where $X_1(0) \dist f.$ The condition
\[
\lim_{t \downarrow 0}q(t, \md x) = f(\md x)
\] can be interpreted as $\wt{X}_1(t)$ approaching $X_1(0)$ almost surely as $t \to 0$, so in particular the distribution of $\wt{X}_1$ at time $t$ converges to
that of $X_1(0)$ as $t \to 0.$ See Remark \ref{rem:change_time_ell_int_rep} below.
\begin{remark}\label{rem:change_time_ell_int_rep}
One can also demonstrate \eqref{eq:integral_rep_ell} from the time change representation of $\wt{X}_1$ given in \eqref{eq:time_change_limit}. Reflected Brownian motion $Z_1$ can be expressed as $Z_1 = L_{Z_1} + B$ where $L_{Z_1}$ is the
local time of $Z_1$ at zero and $B$ is a Brownian motion. Then \eqref{eq:time_change_limit} becomes
\[
\wt{X}_1(t) = X_1(0) + L_{Z_1} \circ \int_0^t\sigma^2(\ell(s))\, \md s + B \circ \int_0^t\sigma^2(\ell(s))\, \md s,
\]
and furthermore $\wt{L}_1(t) = L_{Z_1} \circ \int_0^t\sigma^2(\ell(s))\, \ds$. We use the classical fact that 
\[
\ex L_{Z_1}(u) = \frac{1}{2}\int_0^up(0, s)\, \md s,
\] where $p(x, t)$ is 
the transition density of $Z_1$. See \cite{burdzy2004} for a more general result. Define

\[
u(t) := \int_0^t\sigma^2(\ell(s))\, \md s.
\]
Then
\begin{align*}
\ex \wt{L}_1(t) &= \frac{1}{2}\ex (L_{Z_1}\circ u(t))= \frac{1}{2}\int_0^{t}p(0, u(s))\, \md u(s)\\
&= \frac{1}{2}\int_0^{t}p(0, u(s))\sigma^2(\ell(s))\, \md s = \frac{1}{2}\int_0^{t}q(0, s)\sigma^2(\ell(s))\, \md s,
\end{align*}
here $q(x, s)$ is the transition density of $\wt{X}_1.$\\

If $X_1(0) = 0$, we know from 
classical theory \cite{KaratzasShreve} that $\ex L_{Z_1}(u) = 2\sqrt{2u/\pi}.$ 
Therefore, as in Remark \ref{rem:change_time_ell_int_rep},
\begin{align*}
	\ell^2(t) &= \left(\frac{1}{2}\ex L_{Z_1} \circ u(t) \right)^2\\
	&= \frac{2u(t)}{\pi}\\
	&= \frac{2}{\pi}\int_0^t\sigma^2(\ell(s))\, \md s.
\end{align*}
The integral representation of $\ell$ implies $\ell'$ exists, so taking derivatives of both sides above,
\begin{align}\label{eq:PDE_ell}
	\ell(t)\ell'(t) = \frac{1}{\pi}\sigma^2(\ell(t)), \, \ell(0) = 0.
\end{align}
This fact, which was derived using the explicit formula of the expectation for local time of reflected Brownian motion, shows that the solution for $\ell$ in equations \eqref{intro:eq:PDE1}-\eqref{intro:eq:PDE4} can
be computed from $q$ by solving \eqref{eq:PDE_ell}. However, the solution of $q$ and $\ell$ are dependent since $q$ depends on $\ell$ once
it is found.
\end{remark}

\begin{example}[$\sigma(x) = 1/x^{p/2}$ for $p > 0$] \label{example:pde} Solving \eqref{eq:PDE_ell} when $\sigma(x) = 1/x^{p/2}$,
\[
\ell'(t) = \frac{1}{\pi(\ell(t))^{p + 1}}.
\] So
$\ell(t) = \alpha_pt^{\frac{1}{p + 2}},$ where $\alpha_p = (\frac{p + 2}{\pi})^{(1/(p+2))}.$
The time change of the solution to the classical Neumann problem, or equivalently, the time change
of the reflected Brownian motion, is determined by
\[
u(s) = \int_0^s\sigma^2(\ell(x))\, \md x = \frac{1}{\alpha^p}\int_0^s\frac{\md x}{x^{p/(p + 2)}}
=\beta_ps^{\frac{2p + 2}{p + 2}}
\]
where $\beta_p = \frac{p + 2}{2\alpha_p^p(p + 1)}$.
\end{example}
\noindent
The propagation of chaos result, Theorem \ref{progagation_chaos}, gives existence of a process $\wt{X}_1$ whose transition density exists and satisfies \eqref{intro:eq:PDE1}-\eqref{intro:eq:PDE4}. To show uniqueness, we use a stochastic representation (coupling)
by considering two solutions and expressing each as an appropriate time change of the same
reflected Brownian motion.
\begin{theorem}[Uniqueness]\label{th:uniqueness}
There exists a unique pair $(q, \ell)$ that solves \eqref{intro:eq:PDE1}-\eqref{intro:eq:PDE4} in the classical sense.
\end{theorem}
\begin{proof}
Existence follows from letting $q$ be the transition density of $\wt{X}_1(\cdot)$ and $\ell(t) = \ex \wt{L}_1(t)$, as mentioned above. To demonstrate uniqueness, assume $(q_1, \ell_1), (q_2, \ell_2)$ are two solutions of \eqref{intro:eq:PDE1}-\eqref{intro:eq:PDE4}. Let $(\Omega, \prob, (\F_t)_{t \geq 0})$ be a probability space supporting a Brownian motion $B$. Consider two processes $Z_1, Z_2,$ driven by $B$ (so $Z_i$ are coupled on $(\Omega, \prob)$) solving
\begin{align}\label{eq:Z_i}
Z_i(t) = X_1(0) + L_{Z_i}(t) + \int_0^t\sigma(\ell_i(s))\, \md B(s)
\end{align}
where $L_{Z_i}$ is the local time of $Z_i$ at zero and $X_1(0) \dist f(\md x)$ is independent of $B$. The density of $Z_i$, $q_i(t, x)$, solves \eqref{intro:eq:PDE1}-\eqref{intro:eq:PDE4} with $\ell_i$ in place of $\ell$. The same argument as in Remark \ref{rem:change_time_ell_int_rep} shows $q_i(\cdot, \cdot)$
is the unique solution because it is the specific time change given by $\ell_i.$ As mentioned in Remark \ref{remark:system_existence}, $L_{Z_i}$ is the signed running minimum of 
\[
X_1(0) + \int_0^t\sigma(\ell_i(s))\, \md B(s), \, \ell_i(t) = \ex L_{Z_i}(t).
\]
Because $Z_1, Z_2$ are driven by the same Brownian motion we obtain similar bounds as \eqref{eq:POC_exp_bounds}. By Jensen's inequality and Doob's maximal inequality
\begin{align*}
\|\ell_1 - \ell_2\|_{[0, T]}^2 &= \|\ex(L_{Z_1} - L_{Z_2})\|_{[0, T]}^2\\
&\leq \ex\| L_{Z_1} - L_{Z_2}\|_{[0, T]}^2\\
&\leq \ex \| \int_0^\cdot \sigma(\ell_1(s)) - \sigma(\ell_2(s))\, \md B(s)\|_{[0, T]}^2\\
&\leq 4\ex \int_0^T |\sigma(\ell_1(s)) - \sigma(\ell_2(s))|^2\, \md s.
\end{align*}
By Lipschitz continuity of $\sigma$
\begin{align*}
\|\ell_1 - \ell_2\|_{[0, T]}^2 &\leq K\int_0^T|\ell_1(s) - \ell_2(s)|^2\, \md s \leq K\int_0^T\|\ell_1 - \ell_2\|_{[0, s]}^2\, \md s.
\end{align*}
Gr\"onwall's inequality implies $\|\ell_1 - \ell_2\|_{[0, t]}^2$ is bounded by $|\ell_1(0) - \ell_2(0)|\exp Kt$, which is zero. Therefore $\ell_1 = \ell_2$. From the definition \eqref{eq:Z_i}, we see
\[
\|Z_1 - Z_2\|_{[0, T]} \leq 2\left\| \int_0^t \sigma(\ell_1(s)) - \sigma(\ell_2(s))\, \md B(s)\right\|_{[0, T]} = 0.
 \]
 So $Z_1(t) = Z_2(t)$ for all $t \in [0, T]$, almost surely. Consequently
$q_1 = q_2$ because $q_1$ and $q_2$ are the transition density of the same process.
\end{proof}

\begin{remark}\label{remark:strong_existence}
At the end of the proof of Theorem \ref{th:uniqueness} we showed that $\prob\big(Z_1(t) = Z_2(t) \, \text{ for all } t \in [0, T]\big) = 1$. In other words, we demonstrated strong uniqueness of
processes solving \eqref{th:POC:X_tilde} and \eqref{th:POC:ell}. This was used to show 
uniqueness of the pair $(q, \ell)$ solving \eqref{intro:eq:PDE1}-\eqref{intro:eq:PDE4}. In fact, strong 
existence and uniqueness of the aforementioned process is equivalent to existence and uniqueness 
of the PDE.
\end{remark}

\end{document}